%% file: main.tex
\documentclass[10pt]{article}

\usepackage[plainpages=false,pdfpagelabels,colorlinks=true,linkcolor=mypurple,urlcolor=mypurple,citecolor=mypurple]{hyperref}

\usepackage{tptstyle}

\usepackage{fullpage}

\title{Optimal use of a black-box learner in semiparametric estimation}

\author{Yihong Gu \\ Harvard University}
\date{}

\begin{document}
\begin{titlepage}
\clearpage\maketitle
\thispagestyle{empty}

\begin{abstract}{
\input{0-abstract.tex}
} 
\end{abstract} 

% TABLE OF CONTENTS with page link
%\phantomsection
%\renewcommand{\contentsname}{Table of Contents}
%\hypertarget{toc}{}
%\setcounter{tocdepth}{2}
%\tableofcontents

\thispagestyle{empty}
\end{titlepage}

\maintextrefs
\input{1-text.tex}

\input{2-discussion.tex}

\input{3-proof.tex}

\bibliographystyle{apalike2}
\bibliography{main.bbl}

\newpage
\appendix
\appendixrefs

\end{document}

%% file: 0-abstract.tex
Consider the partial linear model $Y = \mu_0(X) + \beta_0 \cdot T + \varepsilon$ and $T = \pi_0(X) + u$ in the {\it structure-agnostic} setting, where we are blind to the structure $\mu_0$ and $\pi_0$ and estimate the nuisances by a black-box hypothesis class. 
The learnability of the class is characterized by the estimation error $\delta_s$ in the absence of model misspecification and its $L_2$ mis-specification error $\delta_{a, \mu}$ and $\delta_{a, \pi}$ for $\mu_0$ and $\pi_0$, respectively.
We propose a novel estimator of the target linear coefficient $\theta_0 = \beta_0$ with error rate
\[
    \frac{1}{\sqrt{n}} + \delta_{a, \mu} \cdot \delta_{a, \pi} + [\delta_s]^2.
\]
A matching lower bound is also established, implying that this rate is unimprovable. Compared with the product rate yielded by double machine learning (DML), our estimator removes the suboptimal term $\max(\delta_{a, \mu}, \delta_{a, \pi})\cdot \delta_s$ at no extra cost or assumption. 

Building on the underlying insights, which are neither tailored to the one-learner setting nor the partial linear model, we propose {\it Transductive Adversarial Moment-calibrated Editing (TAME)}, which locally edits debiasing weights induced by black-box regression estimates on the inference sample through adversarial conditional moment calibration. TAME can be combined with any initial black-box estimates and can strictly improve on DML guarantees when the nuisance difficulties are imbalanced. We discuss how to fully exploit the advantages introduced by TAME, including the gains from using two learners, the resulting under-smoothing principle for model selection, and extensions to other linear functional estimation problems. 

%% file: 1-text.tex
\section{Introduction}

Modern semiparametric estimation leverages flexible black-box models to estimate nuisance functions, raising the fundamental question: {\it How can we optimally transcribe the nuisance errors of generic black-box learners into the final estimation error of the target parameter}? The dominant paradigm, exemplified by Double/Debiased Machine Learning (DML) \citep{chernozhukov2017double, chernozhukov2018double}, plugs the resulting black-box estimates into a Neyman-orthogonal score, yielding a bound in which the nuisance estimation errors enter multiplicatively. While this multiplicative propagation of nuisance errors is known to be sub-optimal in \emph{structure-aware} settings, for example, the nuisance functions are H{\"o}lder smooth \citep{robins2008higher} or sparse linear models \citep{zhang2014confidence, bradic2022testability}, it was recently shown by \cite{gu2026optimally} that the DML multiplicative error rate can also be \emph{provably} improved in \emph{structure-agnostic} settings with generic black-box models. However, the proposed method therein only answers the above question in the specific regime where the treatment function (or more generally, the Riesz representer for the target linear functional) cannot be consistently estimated, under which only the outcome function comes into play. Building upon and generalizing the insights from \cite{gu2026optimally}, this paper answers the above question when both the outcome function and the treatment function can be estimated efficiently by one black-box learner, and based on this, proposes a post-prediction debiasing procedure that can be applied to arbitrary initial black-box estimates.

\subsection{A canonical example and our new message} \label{subsec:problem}

We begin by studying the partial linear model \citep{engle1986semiparametric, rice1986convergence}, one of the most canonical semiparametric problems and a setting in which the structure-aware literature, including either H{\"o}lder smooth class \citep{donald1994series, carroll1997generalized, ai2003efficient, fan2005profile, chen2007large, robins2008higher} or sparse linear models \citep{zhang2014confidence, van2014asymptotically, javanmard14confidence, cai2017confidence, javanmard2017debiasing}, offers substantial messages. Suppose we have 2$n$ observations $\{(X_i, T_i, Y_i)\}_{i=1}^{2n}$ drawn i.i.d. from the model 
\begin{equation}
\label{eq:intro:model}
\begin{alignedat}{2}
    Y &= \beta_0 \cdot T + \mu_0(X) + \varepsilon &&\qquad \text{with} \qquad \mathbb{E}[\varepsilon|X, T] = 0, \\
    T &= \pi_0(X) + u &&\qquad \text{with} \qquad \mathbb{E}[u|X] = 0.
\end{alignedat} 
\end{equation} Here $X \in \mathbb{R}^d$ is the \emph{covariate}, $Y\in \mathbb{R}$ is the \emph{response}, $T$ is the particular variable of interest, and we refer to it as a \emph{treatment} variable, following the terminology in causal inference. The goal is to estimate the linear coefficient $\theta_0 = \beta_0\in \mathbb{R}$, while $\mu_0$ and $\pi_0$ are \emph{nuisance functions}. We write the sample size as $2n$ to accommodate sample splitting.

Motivated by modern statistical learning, we consider the \emph{structure-agnostic} setting, under which we are blind to any structure of $\mu_0$ and $\pi_0$ like additivity, H{\"o}lder smoothness, or (sparse) linearity, but know that they can be estimated well by a provided black-box machine learning (ML) model, represented as a hypothesis function class $\mathcal{G}$. Following the notations of statistical learning \citep{bartlett2005local, vladimir2006local, massart2006risk, wainwright2019high}, we use three quantities $(\delta_\mu^\appr, \delta_\pi^\appr, \delta^\stoc_n)$ to characterize how the black-box model $\mathcal{G}$ can estimate the nuisance functions. Denote $\|f\|_2 = (\mathbb{E}[f^2(X)])^{1/2}$, we use $\delta^\appr_h:= \inf_{g\in \mathcal{G}} \|g - h_0\|_2$ to denote the approximation error with respect to the ground-truth functions $h \in \{\mu, \pi\}$. We let $\delta_n^\stoc$ be the stochastic error of the function class $\mathcal{G}$ based on $n$ i.i.d. observations. It will be formally defined later and can be interpreted as the $L_2$ regression error when the ground-truth regression function belongs to $\mathcal{G}$. For example, a Vapnik–Chervonenkis (VC) \citep{vapnik1971uniform} class with VC dimension $D=o(n)$ has stochastic error $\delta_{n}^\stoc \lesssim \sqrt{D\log(en/D)/n}$. Applying the standard empirical process arguments, running least squares over $\mathcal{G}$ can estimate a regression function $h_0$ at the error rate $O(\delta_h^\appr + \delta_{n}^\stoc)$. For example, we have 
\[
\|\hat{\pi} - \pi_0\|_2 \lesssim \delta^{\appr}_\pi + \delta^{\stoc}_n \qquad \text{ for }\qquad \hat{\pi} = \argmin_{g\in \mathcal{G}} \frac{1}{n} \sum_{i=1}^n [T_i - g(X_i)]^2.\] 
Here, we explicitly disentangle $\delta^\appr$ and $\delta^\stoc$ in our formalization instead of using the combined $L_2$ estimation error. This is because they affect the estimation of $\beta_0$ differently. In particular, structure-aware methods can improve on DML by reducing the contribution of the stochastic error $\delta^\stoc$; see Section 1.3 of \cite{gu2026optimally} for a detailed discussion. 

This paper characterizes the best attainable performance for estimating the target $\theta_0 = \beta_0$ using a generic black-box ML class $\mathcal{G}$ with known error budgets $(\delta_\mu^\appr, \delta_\pi^\appr, \delta_n^\stoc)$. As a benchmark, DML first uses $\mathcal{G}$ to estimate the two nuisance functions either explicitly \citep{chernozhukov2017double, chernozhukov2018double, chernozhukov2022debiased, foster2023orthogonal} or implicitly \citep{chernozhukov2022automatic, chernozhukov2026adversarial} on one sample split and plugs $(\hat{\mu}, \hat{\pi})$ into a doubly robust \citep{robins1994estimation} estimator on another hold-out split, yielding
\begin{align*}
    |\hat{\theta}_{\mathtt{DML}} - \theta_0| &\lesssim n^{-1/2} + \|\hat{\mu} - \mu_0\|_2 \cdot \|\hat{\pi} - \pi_0\|_2 \\
    &\lesssim n^{-1/2} + \delta_\mu^\appr \cdot \delta_\pi^\appr + [\delta_n^\stoc]^2 + \max\{\delta_\pi^\appr, \delta_\mu^\appr\} \cdot \delta_n^\stoc.
\end{align*} where the second inequality substitutes the approximation and stochastic errors into the $L_2$ estimation error. 

Recently, \cite{gu2026optimally} showed that this rate can be improved for a generic black-box class $\mathcal{G}$. They proposed another estimator with error guarantees $n^{-1/2} + \delta_\mu^\appr + [\delta_n^\stoc]^2$, together with a matching lower bound in the regime of $\delta_\pi^\appr \asymp 1$, under which one cannot consistently estimate the treatment regression function. Their method removes the $\delta_\pi^\appr \cdot \delta_n^\stoc$ at the cost of inflating the bias $\delta_\mu^\appr \cdot \delta_\pi^\appr$ by $\delta_\mu^\appr$, leaving the optimal rate unresolved when $\delta_\pi^\appr = o(1)$. 

This paper proposes a novel method that, for any given black-box class $\mathcal{G}$, attains
\begin{align}
\label{eq:intro-error}
    |\hat{\theta}_{\ourone} - \theta_0| \lesssim n^{-1/2} + \delta_\mu^\appr \cdot \delta_\pi^\appr + [\delta_n^\stoc]^2
\end{align} A lower bound is also proved, showing that the above error rate is not improvable even if $\mathcal{G}$ is specialized as a sparse linear model with identity covariance. The one-learner analysis further reveals a calibration principle that is not specific to either one learner or the partial linear model. This principle underlies the general \our~procedure and the extensions developed later in the paper.

\subsection{New contributions}

\myparagraph{Optimally using a learner in partial linear models.} In \cref{sec:estimator}, we focus on the canonical one-learner partial linear model introduced in \cref{subsec:problem}. We present an estimator attaining the rate \eqref{eq:intro-error}, explain its rationale, and establish the upper bound result to be proved in \cref{sec:proof-main}. A matching lower bound is also established by restricting $\mathcal{G}$ to a sparse linear model with identity covariance. These theoretical results yield the three implications, extending \cite{gu2026optimally} to a broader regime in the \emph{structure-agnostic} setting. 

\begin{itemize}[itemsep=0pt, topsep=0pt]
\item[(a)] Regardless of which nuisance function is harder to learn, one can eliminate all mixed terms that are first-order in stochastic error, i.e., $(\delta_\mu^\appr + \delta_\pi^\appr) \cdot \delta_n^\stoc$, in the DML rate. 
\item[(b)] Consequently, asymptotic normality remains possible even if one nuisance function is only consistently estimated, provided the other is estimated well and appropriate under-smoothing is adopted. 
\item[(c)] In the one-learner regime, which is where most provable structure-aware improvements are proposed, structure agnosticism introduces no fundamental limit: the rate \eqref{eq:intro-error} is attainable for any black-box class $\mathcal{G}$ and remains unimprovable when $\mathcal{G}$ is restricted to a sparse linear class. 
\end{itemize}

The estimator adopts two key ideas. Firstly, it searches for debiasing weights that remain close to those induced by the black-box regression estimates while approximately satisfying the adversarial moment constraints introduced by \cite{gu2026optimally}. Second, when the outcome function is harder to estimate, it constructs the outcome-side debiasing weights by a cross-sample objective. The weights depend only on the covariates in the inference split, while their moment constraints are imposed using the other split. This preserves the required conditional independence and removes the mixed term $\delta_n^\stoc \cdot \delta_\mu^\appr$.

\myparagraph{TAME: A generic fine-tuning debiased procedure. } In \cref{sec:full-method}, we go beyond the one-learner setting and develop {\it\textsf{T}ransductive \textsf{A}dversarial \textsf{M}oment-calibrated \textsf{E}diting (\textsf{TAME})}, a post-prediction debiasing procedure that can be applied to arbitrary initial black-box estimates. We establish a generic oracle-type inequality and then specialize it to the two-learner setting. When the nuisance difficulties are imbalanced and the nuisance contribution dominates $n^{-1/2}$, the two-learner version of \our~attains a strictly faster rate than the DML rate. Its optimal model selection uses the standard bias-and-variance tradeoff for the harder nuisance and under-smooths the easier one. We illustrate these with the neural network learners. 

\myparagraph{Extensions beyond the partial linear model.} In \cref{sec:extension}, we extend the outcome-easier branch (\texttt{O}-branch) to general linear functional estimation and the treatment-easier branch (\texttt{R}-branch) to average treatment effect estimation. These results show that the principles underlying \our~are not limited to the partial linear model.

\subsection{Related works}

Many structure-aware estimation error results for partial linear models can be expressed in our one-learner setting under additional structural assumptions; see the translation of selected results in our language in \cref{tb:rate}. When the class $\mathcal{G} = \mathcal{G}_p^{\mathsf{lin}}$ to be used is a linear model with $p$ basis, so that $\delta_n^\stoc \asymp (p/n)^{1/2}$, \cite{donald1994series} show the joint least squares estimator $\hat{\beta}$ can estimate $\beta_0$ at the error rate $|\hat{\beta} - \beta_0| \lesssim n^{-1/2} + \delta_\mu^\appr \cdot  \delta_\pi^\appr + n^{-1} \sqrt{p}$ provided $p=o(n)$. \cite{robins2008higher} uses the higher-order influence function for estimating a more general nonlinear functional that reduces to $\beta_0$ in \eqref{eq:intro:model} and attains the same error without the requirement $p=o(n)$. 
These linear sieve methods exploit a known $p$-dimensional subspace and its projection operator, a structure that does not directly extend to the non-convex sparse linear model \citep{bradic2019minimax}.
When $\mathcal{G} = \mathcal{G}_{p,s}^{\mathsf{slin}}$ is a sparse linear model of ambient dimension $p$ and sparsity level $s$, its stochastic error is upper bounded by $(s\log(p)/n)^{1/2}$. As a special case of the capped $\ell_1$ approximate sparsity considered by \cite{zhang2014confidence}, the debiased Lasso estimator therein can estimate $\beta_0$ at the error rate $n^{-1/2} + n^{-1} s_\mu \log(p)$ when $\mu_0 \in \mathcal{G}_{p, s_\mu}^{\mathsf{slin}}$ is an exact $s_\mu$-sparse linear model and the $L_2$ projection of $\pi_0$ into the full linear $p$ basis, which is $s_\pi$-sparse, can at least be consistently estimated by Lasso, i.e., $s_\pi \log(p) = o(n)$. On the other hand, without assuming any sparsity conditions on $\mu_0$, the estimator in \cite{bradic2022testability} can estimate $\beta_0$ at the error rate $n^{-1/2} + n^{-1} s_\pi \log(p)$ when $\pi_0 \in \mathcal{G}_{p, s_\pi}^{\mathsf{slin}}$ is an exact $s_\pi$-sparse linear model. Our estimator recovers the above two sharp rates for partial linear model, which is the ``worst-case'' structure-aware models, and generalizes them in two aspects: we allow $\mathcal{G}$ to be any hypothesis class, and also give a tight dependence $\delta_\mu^\appr \cdot \delta_\pi^\appr$ on the misspecification error; see \cref{remark:slm}. 

\begin{table}
\begin{center}
\begin{tabular}{l|l|l|l}
\hline
\hline
Selected reference & $\mathcal{G}$ & Nuisance remainder & Regime\\
\hline
\cite{donald1994series} & Linear & $\delta^\appr_\mu \cdot \delta^\appr_\pi + n^{-1/2} \cdot \delta_n^\stoc$ & $\delta_n^\stoc = o(1)$ \\
\cite{robins2008higher} & Linear & $\delta^\appr_\mu \cdot \delta^\appr_\pi + n^{-1/2} \cdot \delta_n^\stoc$ & - \\
\hline
\cite{zhang2014confidence} & Sparse linear & $(\delta_n^\stoc)^2$ & $\delta_\mu^\appr = 0$\\
\cite{bradic2022testability} & Sparse linear & $(\delta_n^\stoc)^2$ & $\delta_\pi^\appr = 0$\\
\hline
\cite{chernozhukov2017double} & Generic & $ \delta^\appr_\mu \cdot \delta^\appr_\pi + (\delta_n^\stoc)^2 + (\delta^\appr_\mu \lor \delta^\appr_\pi) \cdot \delta_n^\stoc$ & - \\
\cite{gu2026optimally} & Generic & $\delta_\mu^\appr + [\delta_n^\stoc]^2$ & - \\
This paper & Generic & $\delta_\mu^\appr \cdot \delta_\pi^\appr + [\delta_n^\stoc]^2$ & -\\
\hline
\hline
\end{tabular}
\end{center}
\caption{Nuisance remainders from selected results for the partial linear model expressed in our notation; the common $n^{-1/2}$ term and regularity conditions are suppressed. The sparse linear rows display the result for exact $s$-sparse linear models used in our estimation error comparison rather than the full generality of the cited papers.}
\label{tb:rate}
\end{table}

Our constructed estimator combines two different estimators, handling the cases where either the treatment or outcome function is harder to learn separately. The estimator targeted at a harder treatment function is related to the idea of covariate balancing \citep{hellerstein1999imposing, imai2014covariate, zubizarreta2015stable, zhao2019covariate, ning2020robust, athey2018approximate, fan2022optimal}, or more generally, augmented minimax linear estimation for a convex class \citep{hirshberg2021augmented, bruns2025augmented, kong2025asymptotics}. 
The other estimator, which deals with a harder outcome function, uses a cross-sample objective for carefully controlling conditional independence structure and shares a similar spirit with the high-dimensional debiasing constructions of \cite{bradic2022testability, wang2024debiasd}. 
Other provable improvements over the DML benchmark fall broadly into two categories. The first is structure-aware: some methods exploit linear sieve structure \citep{robins2008higher, robins2016technical, liu2017semiparametric}, whereas others exploit sparse high-dimensional structure \citep{athey2018approximate, tan2020model, ning2020robust, wang2024debiasd, celentano2023challenges}. The second category relies on additional assumptions on the first-stage fits \citep{van2014targeted, bonvini2024doubly, van2024doubly}. By contrast, our method uses only the learnability of the black-box class, as characterized by its approximation and stochastic error budgets, to attain provable improvement.

\medskip
\noindent \textbf{Notations.} Denote $x\lor y = \max\{x, y\}$ and $x\land y = \min\{x, y\}$. In this paper, we use $\const$ to denote a universal constant that is related to the boundedness from below and above conditions. We use $a \lesssim b$, $b \gtrsim a$, or $a = O(b)$ if there exists some constant $C$ depending on $\const$ such that $a \le Cb$ for any $n \ge C$, and write $a \asymp b$ if $a\lesssim b$ and $a \gtrsim b$. Denote $a = o(b)$, $a \ll b$, $b\gg a$, if $a$ and $b$ changes with $n$ in that $\lim_{n\to \infty} a/b = 0$. We use $\poly(a) \in \mathbb{R}$ to denote a scalar that depends polynomially on $a$.

\section{Optimal one-learner TAME for the partial linear model}
\label{sec:estimator}

In this section, we propose the one-learner version of \our, referred to as \ourone, for estimating the linear coefficient $\beta_0$ in the partial linear model \eqref{eq:intro:model}. As a proof of concept for the generic \our~framework, \ourone~attains the structure-agnostic optimal rate \eqref{eq:intro-error} in the one-learner setting. 

\cref{subsec:setup} introduces the regularity condition and formally defines the approximation error and stochastic error for the hypothesis class $\mathcal{G}$. The method is then presented in \cref{subsec:method}, accompanied by the underlying rationale in \cref{subsec:rationale}. The upper bound and asymptotic normality results are established in \cref{sec:upper-bound}, with a matching lower bound result in \cref{sec:lb-plm}.

\subsection{Setup}
\label{subsec:setup}

We first impose a standard regularity condition for the partial linear model \eqref{eq:intro:model}. \cref{cond:reg-plm} (a) bounds the coefficient, nuisance functions, and noises, while (b) ensures that the residual $u=T-\pi_0(X)$ has nondegenerate variance, which allows $\beta_0$ to be identifiable.

\begin{condition}[Regularity conditions for partial linear model] 
\label{cond:reg-plm} There exists a constant $\const\ge 1$ such that:
    \begin{itemize}
    \item[(a)] The coefficient, nuisance functions, and noises are uniformly bounded: $|\beta_0| \le \const$, $\|\mu_0\|_\infty \lor \|\pi_0\|_\infty \le \const$, $|\varepsilon| \lor |u| \le \const$.
    \item[(b)] The treatment variable $T$ cannot be fully explained by $X$, i.e., $\mathbb{E}[u^2] \ge 1/\const$.
    \end{itemize}
\end{condition}

To formalize the stochastic error of a function class, we first define its localized population Rademacher complexity.

\begin{definition}[Localized population Rademacher complexity]
For a given radius $\delta>0$, function class $\mathcal{F}$ on $\mathcal{Z}$, and a probability distribution $\nu$ on $\mathcal{Z}$, define
\begin{align*}
    \mathsf{R}_{n,\nu}(\delta;\mathcal{F}) = \mathbb{E}_{Z,R}\left[\sup_{h\in \mathcal{F}, \|h\|_{L_2(\nu)} \le \delta} \left|\frac{1}{n} \sum_{i=1}^n R_i h(Z_i)\right|\right],
\end{align*} where $Z_1,\ldots, Z_n \in \mathcal{Z}$ are $n$ i.i.d. observations from $\nu$, and $R_1,\ldots, R_n$, independent of $Z_1,\ldots, Z_n$, are i.i.d. Rademacher variables taking values in $\{-1,+1\}$ with equal probability.
\end{definition} 

We next use this complexity to define the stochastic error of a uniformly bounded hypothesis class.
%We assume the ``black-box model'' $\mathcal{G}$ we used is uniformly bounded and has stochastic error upper bounded by $\delta_n^\stoc$ and encapsulate it as a definition for ease of use for the remaining sections. 
\begin{definition}[Hypothesis class with stochastic error $\bar{\delta}$] \label{def:hcs} 
Let $\mathcal{F}$ be a function class on domain $\mathcal{Z}$, and let $\nu$ be a probability distribution on $\mathcal{Z}$. Define
\begin{align}
     \delta^\stoc(\mathcal{F}, n, \nu) &:= \inf\left\{\delta_s>\sqrt{\frac{\log(n)}{n}}: \sup_{\delta \ge \delta_s} \frac{\mathsf{R}_{n,\nu}(\delta; \partial \mathcal{F})}{\const \cdot \delta} \le \delta_s\right\},
\label{eq:error-stoc-f}
\end{align} where $\partial \mathcal{F} := \{f - \tilde{f}: f, \tilde{f} \in \mathcal{F}\}$. We say $\mathcal{F}$ is a uniformly bounded class with stochastic error upper bounded by $\bar{\delta}$, denoted by $\mathcal{F} \in \mathcal{HCS}(\bar{\delta}; n, \nu)$, if $\|f\|_\infty \le \const$ for any $f\in \mathcal{F}$ and $\delta^\stoc(\mathcal{F}, n, \nu) \le \bar{\delta}$.
\end{definition}

\begin{condition}[Hypothesis class $\mathcal{G}$] \label{cond:g} Let $\nu_x$ be the distribution of $X$. Assume $\mathcal{G} \in \mathcal{HCS}(\delta_n^\stoc; n, \nu_x)$. Moreover, we define the approximation errors, i.e., $L_2$ misspecification errors, as 
\begin{align}
    \delta_{\mu}^\appr &:= \inf_{g\in \mathcal{G}} \|g - \mu_0\|_{L_2(\nu_x)} \qquad \text{and} \qquad  \delta_\pi^\appr = \inf_{g\in \mathcal{G}} \|g - \pi_0\|_{L_2(\nu_x)}
\label{eq:error-appr-g}.
\end{align}
\end{condition}

\begin{remark} For ease of presentation, we adopt three simplifications: (1) the coefficient, nuisance functions, noises, and hypothesis classes are uniformly bounded; (2) the local Rademacher complexity satisfies the bounds for any $\delta \ge \delta_n^\stoc$; and (3) the critical radius is at least $\sqrt{\log(n)/n}$.
The boundedness assumptions in (1) can be replaced by suitable sub-Gaussian tail conditions. For (2), if
$\bar{\mathcal{F}} = \partial\mathcal F$ is star-shaped, i.e., $\alpha f \in \bar{\mathcal{F}}$ for any $f\in \bar{\mathcal{F}}$ and $|\alpha| \le 1$, then \eqref{eq:error-stoc-f} reduces to the usual critical radius definition; see \citet{wainwright2019high}. Without star-shapedness, one may instead use the usual critical radius at the cost of slightly more complicated final bounds. (3) avoids repeatedly displaying negligible $\log(n)/n$ terms.
\end{remark}

\subsection{The debiased estimator for one learner}
\label{subsec:method}

For the observed data $\mathcal{D} = \{(X_i, T_i, Y_i)\}_{i=1}^{2n}$ with splits $\mathcal{D}_1 = \{(X_i, T_i, Y_i)\}_{i=1}^n$ and $\mathcal{D}_2 = \{(X_i, T_i, Y_i)\}_{i=n+1}^{2n}$, we first use $\mathcal D_2$ to obtain initial estimates. In particular, we run joint least squares for $(\beta_0,\mu_0)$ and least squares for $\pi_0$
\begin{align}
\label{eq:est1-step1}
\tag{S0}
\begin{split}
    (\hat{\beta},\hat{\mu}) \in \argmin_{(\beta,g) \in [-\const, \const]\times  \mathcal{G}} \frac{1}{n} \sum_{i=n+1}^{2n} \left(Y_i - \beta \cdot T_i - g(X_i)\right)^2, ~ 
    \hat{\pi} \in \argmin_{g\in \mathcal{G}} \frac{1}{n} \sum_{i=n+1}^{2n} \left(T_i - g(X_i)\right)^2,
\end{split}
\end{align} to get an initial estimate of $(\beta_0,\mu_0)$ and $\pi_0$, respectively. Using the standard empirical process argument, we have $\|\hat{\mu} - \mu_0\|_2 + |\hat{\beta} - \beta_0|\lesssim \delta_\mu^\appr + \delta_n^\stoc$ and $\|\hat{\pi} - \pi_0\|_2 \lesssim \delta_\pi^\appr + \delta_n^\stoc$ with high probability. 

%\begin{remark}
%The procedure till now is identical to that of the standard DML estimator, which ends up using the orthogonal score estimator \eqref{eq:est-dml} for estimating $\theta_0 = \beta_0$. Our method uses $\hat{\mu}, \hat{\pi}$ as an initial estimate and further performs a careful debiasing procedure as illustrated below. 
%\end{remark}

Our estimator requires knowledge of the ordering between $(\delta^\appr_\mu, \delta^\appr_\pi)$, and selects between the two estimators as follows:
\begin{align}
\label{eq:est1-est}
\hat{\theta} = \begin{cases}
    \hat{\theta}_\mu ~\text{in \eqref{eq:est1-case1-step3}}~ & \qquad \text{if} ~~\delta^\appr_\mu \le \delta^\appr_\pi,\\
    \hat{\theta}_\pi ~\text{in \eqref{eq:est1-case2-step3}}~ & \qquad \text{if} ~~\delta^\appr_\mu > \delta^\appr_\pi
\end{cases}
\end{align}
We let $\lambda> 0$ denote the tuning parameter in the minimax programs below; it should be tuned according to the full error budgets $(\delta_\mu^{\mathrm{appr}}, \delta_\pi^{\mathrm{appr}}, \delta_n^{\mathrm{stoc}})$ to attain the optimal rate.
\medskip

\noindent \underline{\sc Case 1: $\delta^\appr_\mu \le \delta^\appr_\pi$.}  Let $\tilde{a} = (\tilde{a}_1,\ldots, \tilde{a}_n)\in \mathbb{R}^n$ with $\tilde{a}_i =  [T_i - \hat{\pi}(X_i)] / \allowbreak [\{\frac{1}{n} \sum_{l=1}^n T_{l} (T_{l} - \hat{\pi}(X_{l}))\} \lor (2\const)^{-1}]$ be the debiasing weights induced by the black-box estimates. The lower truncation keeps the denominator in $\tilde{a}_i$ bounded away from zero. Here we first use $\mathcal D_1$ to construct empirical debiasing weights $\hat a=(\hat a_1,\ldots,\hat a_n)$ through the minimax program
\begin{align}
\label{eq:est1-case1-step2}
\tag{S1-$\mu$}
\begin{split}
    \hat{a} \in \argmin_{a\in \mathbb{R}^n} &~ \lambda \cdot Q(a, \tilde{a}) + \sup_{\beta \in \mathbb{R}, f\in \partial \mathcal{G}} \Phi_\mu(f, \beta, a) ~~~\text{where}~~ Q(v, v')=\frac{1}{n}\sum_{i=1}^n (v_i - v_i')^2,\\
    & \qquad \Phi_\mu(f, \beta, a) = \left| \frac{1}{n}\sum_{i=1}^n (\beta \cdot T_i + f(X_i)) \cdot a_i - \beta\right| - \frac{1}{n} \sum_{i=1}^n f^2(X_i).
\end{split} 
\end{align} For definiteness, if the objective in \eqref{eq:est1-case1-step2} does not admit a finite minimizer, we set \(\widehat a=0\). Conditional on the initial estimates from $\mathcal D_2$, $\hat a$ depends on $\mathcal{D}_1$ only through $\{(X_i,T_i)\}_{i=1}^n$, and not through $\{Y_i\}_{i=1}^n$. The estimate of $\theta_0 = \beta_0$ in this case is the following debiased estimator
\begin{align}
\label{eq:est1-case1-step3}
\tag{S2-$\mu$}
\hat{\theta}_\mu = \frac{1}{n} \sum_{i=1}^n \left(Y_i - \hat{\mu}(X_i)\right) \cdot \hat{a}_i. 
\end{align}

\noindent \underline{\sc Case 2: $\delta^\appr_\mu > \delta^\appr_\pi$.} Let $\tilde{m} = (\tilde{m}_1,\ldots, \tilde{m}_n) \in \mathbb{R}^n$ with $\tilde{m}_i = \hat{\mu}(X_i) + \hat{\beta} \cdot \hat{\pi}(X_i)$ be the debiasing weights induced by the black-box estimates in this case. We will use both $\mathcal D_1$ and $\mathcal{D}_2$ to construct empirical debiasing weights $\hat m=(\hat m_1,\ldots,\hat m_n)$ on $\mathcal{D}_1$ through the following minimax program,  
\begin{align}
\label{eq:est1-case2-step2}
\tag{S1-$\pi$}
\begin{split}
    \hat{m} \in \argmin_{m\in \mathbb{R}^n} &~\lambda \cdot Q(m, \tilde{m}) + \sup_{f \in \partial \mathcal{G}} \Phi_\pi(f, m), ~~~ \text{where}\\
    &\qquad ~\Phi_\pi(f, m) = \left|\frac{1}{n} \sum_{i=n+1}^{2n} Y_i f(X_i) - \frac{1}{n} \sum_{i=1}^n m_i f(X_i) \right| - \frac{1}{n} \sum_{i=1}^n f^2(X_i)
\end{split}
\end{align} Thus, $\widehat m$ depends on all of $\mathcal{D}_2$ and, conditional on $\mathcal{D}_2$, on $\mathcal{D}_1$ only through the covariates $\{X_i\}_{i=1}^n$. This differs from \eqref{eq:est1-case1-step2}, where $\hat a$ may also depend on the treatments $\{T_i\}_{i=1}^n$ in $\mathcal{D}_1$. This \emph{cross-sample} construction preserves the conditional mean-zero structure needed for the treatment noises $\{u_i\}_{i=1}^n$. The estimate of $\theta_0 = \beta_0$ in this case is the following debiased estimator
\begin{align}
\label{eq:est1-case2-step3}
\tag{S2-$\pi$}
\hat{\theta}_\pi = \left\{(2\const)^{-1} \lor \frac{1}{n} \sum_{i=1}^n [T_i - \hat{\pi}(X_i)]^2\right\}^{-1}\frac{1}{n} \sum_{i=1}^n \left(Y_i - \hat{m}_i\right) \left(T_i - \hat{\pi}(X_i)\right). 
\end{align} 

\subsection{Rationale}
\label{subsec:rationale}

To keep the presentation concise in this subsection, the inequalities involving ``$\lesssim$'' are understood to hold with high probability. Denote $\|v\|_n = (\frac{1}{n} \sum_{i=1}^n v_i^2)^{1/2}$ for an $n$-dimensional vector $v=(v_1,\ldots,v_n)\in\mathbb R^n$. We identify a function $f$ with its evaluation vector $f_{1:n}:=(f(X_1),\ldots,f(X_n))$ and accordingly write $\|f\|_n:=\|f_{1:n}\|_n$.

We first explain the rationale behind the estimator $\hat{\theta}_\mu$ by considering the regime where the denominator in $\tilde{a}_i$ exceeds $1/(2\const)$ so the lower bound truncation is inactive. This event occurs with high probability when $\hat{\pi}$ consistently estimates $\pi_0$; the formal analysis also covers the truncated case. 

At a high level, it uses the adversarial moment-calibration criterion from \cite{gu2026optimally} to fine-tune the debiasing weights in a neighbourhood of the weights $\tilde{a}$ used by the standard DML estimator, so that the resulting estimator benefits from both objectives. To see this, when $\lambda \to \infty$, we will have $\hat{a}_i \to \tilde{a}_i = (T_i - \hat{\pi}(X_i)) / \{\frac{1}{n} \sum_{l=1}^n T_l(T_l - \hat{\pi}(X_l))\}$, which coincides with the double machine learning estimator \eqref{eq:est-dml} when the weights $\hat{a}$ are plugged into \eqref{eq:est1-case1-step3}; when $\lambda \to 0$, it is similar to the minimax empirical debiasing estimator in \cite{gu2026optimally}, which is designed specifically to remove the first-order stochastic error at the cost of replacing $\delta_\mu^\appr \cdot \delta_\pi^\appr$ by $\delta_\mu^\appr$. In particular, our construction of the debiasing weights in \eqref{eq:est1-case1-step2} is akin to the following constrained formulation obtained via the Lagrange multiplier method
\begin{align}
\begin{split}
    \hat{a} \in &\argmin_{a \in \mathcal{A}} \frac{1}{n} \sum_{i=1}^n (a_i - \tilde{a}_i)^2 ~~ \text{where} \\
    & \mathcal{A} = \Bigg\{ a \in \mathbb{R}^n: \frac{1}{n} \sum_{i=1}^n a_i T_i = 1, \sup_{f \in \partial \mathcal{G}} \left|\frac{1}{n} \sum_{i=1}^n f(X_i) \cdot a_i \right|
    - \frac{1}{n} \sum_{i=1}^n f^2(X_i) \le \tilde{C} (\delta^\stoc_n)^2\Bigg\}
\end{split}
\label{eq:plm-constrained-mu}
\end{align} for a large enough constant $\tilde{C} = \poly(\const)$. The moment inequality in $\mathcal{A}$ can be viewed as a black-box analogue of a Dantzig-type \citep{candes2007dantzig} feasibility constraint. 

Now, let us briefly explain why the weights in \eqref{eq:plm-constrained-mu}, plugged into \eqref{eq:est1-case1-step3} can yield the upper bound $n^{-1/2} + \delta_\mu^\appr \cdot \delta_\pi^\appr + [\delta_n^\stoc]^2$ when $\delta_\mu^\appr \le \delta_\pi^\appr$. Let $\bar{D} = \frac{1}{n}\sum_{l=1}^n T_l (T_l - {\pi}_0(X_l))$. Since $\bar{D}$ is an average of i.i.d. random variables with mean $\mathbb{E}[u^2] \ge 1/\const$, we have $\bar{D} \gtrsim 1$ for large $n$. Moreover, the ``oracle weights'' $\bar{a} \in \mathbb{R}^n$ with $\bar{a}_i = [T_i - \pi_0(X_i)]/ \bar{D} = u_i / \bar{D}$ lie in the feasible set $\mathcal{A}$ in \eqref{eq:plm-constrained-mu} with high probability; see the detailed justification in Section 2.2 of \cite{gu2026optimally}. Thus, the feasible set $\mathcal{A}$ is non-empty with high probability, and its constraints apply to the solution $\hat{a}$. 

We pick fixed $\bar{g}_\mu \in \mathcal{G}$ such that $\|\bar{g}_\mu - \mu_0\|_2 \le 2\delta_\mu^\appr$. It follows from the model \eqref{eq:intro:model} and the constraint $\frac{1}{n} \sum_{i=1}^n \hat{a}_i T_i = 1$ that
\begin{align*}
    \hat{\theta}_\mu - \theta_0 &= \frac{1}{n} \sum_{i=1}^n (\varepsilon_i + \beta_0 \cdot T_i + \mu_0(X_i) - \hat{\mu}(X_i)) \cdot \hat{a}_i - \beta_0 \\
    &= \underbrace{\frac{1}{n} \sum_{i=1}^n \varepsilon_i \hat{a}_i}_{\mathsf{T}_{\mu, n}} + \underbrace{\frac{1}{n} \sum_{i=1}^n \left(\mu_0 - \bar{g}_\mu\right)(X_i) \cdot \hat{a}_i}_{\mathsf{T}_{\mu, a}} + \underbrace{\frac{1}{n} \sum_{i=1}^n \left(\bar{g}_\mu-\hat{\mu}\right)(X_i) \cdot \hat{a}_i}_{\mathsf{T}_{\mu, s}}. 
\end{align*}
As shown below $\|\hat{a}\|_n \lesssim 1$. Moreover, conditioned on $\mathcal{D}_2$ and $\{(X_i, T_i)\}_{i=1}^n$, the weights $\hat{a}$ are fixed and $\varepsilon_i$ are independent with conditional zero mean. Hence $|\mathsf T_{\mu, n}|\lesssim n^{-1/2}$. 

The minimization of the empirical $L_2$ error between our weights $\hat{a}$ and the weights $\tilde{a}$ given by the black-box estimate $\hat{\pi} \in \mathcal{G}$ will help minimize the error in $\mathsf{T}_{\mu, a}$. To see this, by the minimization objective in \eqref{eq:plm-constrained-mu}, we have $\|\hat{a} - \tilde{a}\|_n \le \|\bar{a} - \tilde{a}\|_n$. It is easy to verify that $\|\tilde{a} - \bar{a}\|_n \lesssim \|\hat{\pi} - \pi_0\|_2 \lesssim \delta_n^\stoc + \delta_\pi^\appr$. This means our debiasing weights $\hat{a}$ are also close to the ``oracle weights'' $\bar{a}$ in a similar manner to $\tilde{a}$: $\|\hat{a} - \bar{a}\|_n \le \|\hat{a} - \tilde{a}\|_n + \|\bar{a} - \tilde{a}\|_n \le 2\|\bar{a} - \tilde{a}\|_n \lesssim \delta_n^\stoc + \delta_\pi^\appr$. Since $\{u_i\}_{i=1}^n$ are independent zero-mean random variables conditioned on $\{X_i\}_{i=1}^n$, it thus follows from the triangle inequality and the Cauchy-Schwarz inequality that
\begin{align*}
    |\mathsf{T}_{\mu, a}| &\le \left|\frac{1}{n} \sum_{i=1}^n (\mu_0 - \bar{g}_\mu)(X_i) \cdot u_i\right|  / |\bar{D}| + \|\mu_0 - \bar{g}_\mu\|_n \cdot \|\hat{a} - \bar{a}\|_n \\
    &\lesssim \frac{1}{\sqrt{n}} + \delta_\mu^\appr\cdot\left( \delta_\pi^\appr + \delta^\stoc_n \right) \lesssim \frac{1}{\sqrt{n}} + \delta_\mu^\appr \cdot \delta_\pi^\appr + [\delta^\stoc_n ]^2,
\end{align*} where the last inequality follows from $\delta_\mu^\appr \cdot \delta_n^\stoc \le [\delta_\mu^\appr]^2 + [\delta_n^\stoc]^2 \le \delta_\mu^\appr \cdot \delta_\pi^\appr + [\delta_n^\stoc]^2$ as $\delta_\mu^\appr \le \delta_\pi^\appr$. 

The adversarial moment constraint with respect to $\partial \mathcal{G}$ in $\mathcal{A}$ will make the term $\mathsf{T}_{\mu, s}$ scale quadratically in $(\delta_\mu^\appr + \delta_n^\stoc)$, which explicitly removes the $\delta_n^\stoc \cdot \delta_\pi^\appr$ error in the DML product rate when $\delta_\pi^\appr \gg \delta_\mu^\appr$. To see this, given $\hat a\in\mathcal A$, the following instance-dependent error bound holds
\begin{align} \label{eq:inst-case1}
    \forall g, \tilde{g} \in \mathcal{G}, \qquad \left|\frac{1}{n} \sum_{i=1}^n (g - \tilde{g})(X_i) \cdot \hat{a}_i \right| \le \|g - \tilde g\|_n^2 + \tilde{C} (\delta_n^\stoc)^2. 
\end{align} Applying this with $(g,\tilde{g})=(\bar{g}_\mu, \hat \mu)$ gives $|\mathsf{T}_{\mu,s}| \lesssim \|\hat{\mu} - \bar{g}_\mu\|_n^2 + (\delta_n^\stoc)^2 \lesssim [\delta_\mu^\appr + \delta_n^\stoc]^2$. The squared approximation error $[\delta_\mu^\appr]^2$ is absorbed into $\delta_\mu^\appr \cdot \delta_\pi^\appr$ given the regime $\delta_\mu^\appr \le \delta_\pi^\appr$ under consideration. Combining the error bounds on $\mathsf{T}_{\mu,n}$, $\mathsf{T}_{\mu,a}$, and $\mathsf{T}_{\mu, s}$ yields the stated rate. 

\begin{remark}[Combining two objectives] For linear \citep{bruns2025augmented} and sparse linear classes \citep{zhang2014confidence, van2014asymptotically}, regression risk minimization can be equivalent to imposing the moment constraints with respect to the adopted hypothesis class, by the minimax duality \citep{hirshberg2021augmented} and the KKT condition \citep{bickel2009simutaneous}, respectively. Such an equivalence need not hold for a generic, possibly non-convex hypothesis class $\mathcal{G}$. Our estimator thus combines the centered ridge penalty $Q(\hat{a}, \tilde{a})$ for deviating from the regression-induced weights with an adversarial moment constraint, rather than relying on either objective to imply the other. 
\end{remark}

For {\sc Case 2}, we will also adopt the idea of using flexible weights to eliminate the potential $\delta_\mu^\appr \cdot \delta_n^\stoc$ error term, but an additional challenge occurs here: the outcome-side debiasing weights must be constructed without using the treatment residuals in the inference sample $\mathcal{D}_1$. Otherwise, the term involving $u_i$ is no longer conditionally mean zero. At first glance, this requirement conflicts with using the outcomes $Y$ in the inference sample $\mathcal{D}_1$ to impose the desired moment constraints.

To see why the construction in {\sc Case 1} cannot be directly mirrored, let us consider the natural DML-form estimator $\hat{\theta} = \frac{1}{n} \sum_{i=1}^n (Y_i - m_i)(T_i - \hat{\pi}(X_i)) / \frac{1}{n} \sum_{i=1}^n T_i (T_i - \hat{\pi}(X_i))$ with the candidate debiasing weights $m=(m_1,\ldots,m_n)\in\mathbb R^n$. Picking a fixed $\bar{g}_\pi \in \mathcal{G}$ such that $\|\bar{g}_\pi - \pi_0\|_2 \le 2 \delta_\pi^\appr$ and expanding $\hat{\theta} - \theta_0$, we need to control the two terms below that share similar spirits to $\mathsf{T}_{\mu, s}$ and $\mathsf{T}_{\mu, n}$, respectively, in {\sc Case 1}:
\begin{align*}
    \underbrace{\frac{1}{n} \sum_{i=1}^n (\mu_0(X_i) - m_i) \left(\bar{g}_\pi - \hat{\pi}\right)(X_i)}_{\breve{\mathsf{T}}_{\pi, s}} + \underbrace{\frac{1}{n} \sum_{i=1}^n (\mu_0(X_i) - m_i) \cdot u_i}_{\breve{\mathsf{T}}_{\pi, n}} 
\end{align*}
One could calibrate $m$ on $\mathcal{D}_1$ to control $\breve{\mathsf{T}}_{\pi,s}$. A same-sample analogue would impose the moment constraints using the residualized outcomes $Y - \hat{\beta} T$. These residualized outcomes will depend on $u_i$ through $T_i = \pi_0(X_i) + u_i$, making $m$ depend on the treatment residuals $\{u_i\}_{i=1}^n$ in $\mathcal{D}_1$. Consequently, $\breve{\mathsf{T}}_{\pi,n}$ would no longer be an average of conditionally zero-mean random variables.

Our construction avoids this conflict by changing both the target of the outcome-side values and the sample used for calibration. The optimized debiasing weight $\hat{m}_i$ will instead approximate $\bar{m}_i =\mathbb{E}[Y|X_i]= \mu_0(X_i)+\beta_0\pi_0(X_i)$, and the required moment constraints are learned from $\mathcal D_2$, while $\hat m$ depends on $\mathcal D_1$ only through its covariates $\{X_i\}_{i=1}^n$. The high-level ideas behind the design of the objective \eqref{eq:est1-case2-step2} are twofold. First, the empirical $L_2$ norm $\|\bar{m} - \hat{m}\|_n$ should be small. This is done by minimizing $\|\hat{m} - \tilde{m}\|_n$ using the initial estimate $\tilde{m}_i = \hat{\beta} \cdot \hat{\pi}(X_i) + \hat{\mu}(X_i)$ satisfying $\|\tilde{m} - \bar{m}\|_n \lesssim \delta_\mu^\appr + \delta_n^\stoc$. Second, minimizing $\sup_{f \in \partial \mathcal{G}} \Phi_\pi(f, m)$ w.r.t. $m$ will impose the instance-dependent error bound $|\frac{1}{n} \sum_{i=1}^n (\bar{m}_i - \hat{m}_i)(g - \tilde{g})(X_i)| \lesssim \|g - \tilde{g}\|_n^2 + (\delta_n^\stoc)^2$ for any $g, \tilde{g} \in \mathcal{G}$, analogous to \eqref{eq:inst-case1}. 

After separating the semiparametric efficient term and the standard higher-order terms arising from the estimated treatment regression and denominator, the three additional dominant terms are
\begin{align*}
    \Bigg|\underbrace{\frac{1}{n} \sum_{i=1}^n (\bar{m}_i - \hat{m}_i) \cdot u_i }_{\mathsf{T}_{\pi, n}} \Bigg| + \Bigg| \underbrace{\frac{1}{n} \sum_{i=1}^n \left(\bar{m}_i - \hat{m}_i\right) \cdot (\pi_0 - \bar{g}_\pi)(X_i) }_{\mathsf{T}_{\pi, a}} \Bigg| + \Bigg|\underbrace{\frac{1}{n} \sum_{i=1}^n \left(\bar{m}_i - \hat{m}_i\right) \cdot (\bar{g}_\pi - \hat{\pi})(X_i) }_{\mathsf{T}_{\pi, s}} \Bigg|.
\end{align*}
Similarly, the minimization of $\|\hat{m} - \tilde{m}\|_n$ will make $|\mathsf{T}_{\pi, a}| \lesssim n^{-1/2} + \delta_\mu^\appr \cdot \delta_\pi^\appr + [\delta^\stoc_n ]^2$, while the minimization of adversarial objective controls the term $|\mathsf{T}_{\pi, s}| \lesssim [\delta_\pi^\appr + \delta_n^\stoc]^2$, which explicitly eliminates the sub-optimal term $\delta_\mu^\appr \cdot \delta_n^\stoc$ in the DML rate. At the same time, the \emph{cross-sample} construction of the weights $\{\hat{m}_i\}_{i=1}^n$ make them dependent on the split $\mathcal{D}_1$ only through $\{X_i\}_{i=1}^n$, thus $\mathsf{T}_{\pi, n}$ continues to be the sum of independent zero-mean random variables conditioned on fixed $(\{X_i\}_{i=1}^n, \mathcal{D}_2)$ in a similar spirit to $\mathsf{T}_{\mu, n}$ for $\hat{\theta}_\mu$. 

\subsection{Upper bound and asymptotic normality}
\label{sec:upper-bound}

Recall the partial linear model in \eqref{eq:intro:model}, and define the first-order oracle term as 
\begin{align}
\label{eq:psi-plm}
    \hat{\varphi}_{\mathsf{PLM}} = \frac{1}{n} \sum_{i=1}^n \varepsilon_i \cdot w_i \qquad \text{with} \qquad w_i = u_i / \mathbb{E}[u^2]
\end{align} Under the additional conditional homoskedasticity assumption $\mathbb{E}[\varepsilon^2\mid X,T]\equiv \sigma_\varepsilon^2$, this is the empirical average of the semiparametric efficient influence function.

The following theorem provides an oracle-type inequality for any given hyperparameter $\lambda \in (0, 1]$.

%\begin{blackframe}
\begin{theorem}[Oracle-type inequality]\label{thm:main-plm}
Consider the partial linear model \eqref{eq:intro:model} satisfying \cref{cond:reg-plm} and the black-box model class $\mathcal{G}$ satisfying \cref{cond:g}. Then there exists a constant $\tilde{C} = \poly(\const)$, if $\delta_n^\stoc \cdot \tilde{C} \le 1$, and $t\in [1, n/\tilde{C}]$, then with probability at least $1-2e^{-t}-2e^{-n (\delta_n^\stoc)^2}$, our estimator $\hat{\theta}$ in \eqref{eq:est1-est} with $\lambda \in [(\delta_n^\stoc)^2, 1]$ satisfies
\begin{align*}
    \tilde{C}^{-1} |\hat{\theta} - \theta_0 - \hat{\varphi}_{\mathsf{PLM}}| \le \delta_\mu^\appr \cdot \delta_\pi^\appr + (\delta_n^\stoc)^2 + \left(\lambda \cdot (\delta_{\max}^\appr)^2 + \frac{\delta_n^\stoc \cdot \delta_{\min}^\appr}{\sqrt{\lambda}}\right) + \mathsf{Res},
\end{align*} where $\delta_{\max}^\appr = \delta_\mu^\appr \lor \delta_\pi^\appr$, $\delta_{\min}^\appr = \delta_\mu^\appr \land \delta_\pi^\appr$, and 
\begin{align*}
    \mathsf{Res} := \sqrt{\frac{t}{n}} \cdot \left[\delta_{\max}^\appr + \frac{\delta_n^\stoc}{\sqrt{\lambda}} + \sqrt{\frac{t}{n}} \right]
\end{align*}is a high-order remainder: for fixed $t\ge 1$, $\mathsf{Res}=o(n^{-1/2})$ as long as $\delta_{\max}^\appr = o(1)$ and $\lambda \gg (\delta_n^\stoc)^2$.
\end{theorem} 
%\end{blackframe}

\cref{thm:main-plm} will be proved in \cref{sec:proof-main}. The two terms dependent on $\lambda$ quantify the trade-off between keeping the edited weights close to the initial black-box estimates and enforcing the adversarial moment constraints. Choosing $\lambda$ as in \cref{coro:error-plm} balances the two terms and gives the desired error bound \eqref{eq:intro-error}. Under the stronger conditions in the corollary, the same choice also yields asymptotic normality. 
\begin{corollary}\label{coro:error-plm}
Under \cref{cond:reg-plm} and \cref{cond:g}, for any $t\ge 1$, with probability at least $1-3e^{-t}-2e^{-n (\delta_n^\stoc)^2}$,
\begin{align*}
    |\hat{\theta} - \theta_0| \lesssim \sqrt{\frac{t}{n}} + \delta_\mu^\appr \cdot \delta_\pi^\appr + (\delta_n^\stoc)^2 \qquad \text{with} \qquad \lambda = \left(\frac{\delta_n^\stoc}{\delta_{\max}^\appr + \delta_n^\stoc}\right)^2.
\end{align*} 
Furthermore, if $\delta_\mu^\appr \cdot \delta_\pi^\appr = o(n^{-1/2})$, $\delta_n^\stoc = o(n^{-1/4})$, and $\delta^\appr_{\max} = o(1)$, then the same choice of $\lambda$ gives $\hat{\theta} - \theta_0 = \hat{\varphi}_{\mathsf{PLM}} + o_P(n^{-1/2})$. 
\end{corollary}

\begin{remark}[Relation to structure-aware results for sparse linear models] \label{remark:slm}
As a special case, if $\mu_0$ is a well-specified sparse linear model with sparsity level $s_\mu$  and $\mathcal{G}$ is the corresponding sparse linear class, then $\delta_n^\stoc = \sqrt{s_\mu \log(p)/n}$ and $\delta_\mu^\appr=0$. Thus, the above error rate is instantiated as $n^{-1/2} + 0 \cdot \delta_\pi^\appr + (\delta_n^\stoc)^2 \asymp n^{-1/2} + s_\mu \log(p)/n$, matching that given by the debiased-Lasso estimator \citep{zhang2014confidence, van2014asymptotically}. On the other hand, if $\pi_0$ is a well-specified sparse linear model with sparsity level $s_\pi$, and $\mathcal{G}$ is the corresponding sparse linear class, then $\delta_\pi^\appr=0$ and $\delta_n^\stoc = \sqrt{s_\pi \log(p) / n}$, and the error bound becomes $n^{-1/2} + s_\pi \log(p) / n$, which coincides with the rate in \cite{bradic2022testability}. 

Thus, \cref{coro:error-plm} extends the previous structure-aware results for the partial linear model in two aspects. First, it applies to any black-box hypothesis class $\mathcal{G}$ with a controlled stochastic error budget, including classes that cannot be reduced to a well-conditioned sparse linear model via the $\epsilon$-cover argument. Second, it explicitly gives the tight dependence $\delta_\mu^\appr \cdot \delta_\pi^\appr$ on the approximation error budgets; \cref{sec:lb-plm} shows that this dependence is unimprovable. 
\end{remark}

When $\delta_\pi^\appr \asymp 1$, the upper bound reduces to the result of \cite{gu2026optimally}. When $\delta_\mu^\appr \asymp 1$, it gives the symmetric rate
\begin{align*}
    n^{-1/2} + \delta_\pi^\appr + [\delta_n^\stoc]^2.
\end{align*} Thus, even if $\mu_0$ cannot be consistently estimated, $\beta_0$ can still be estimated at this rate when $\pi_0$ can be estimated by the hypothesis class $\mathcal{G}$ with approximation error $\delta_\pi^\appr$ and stochastic error $\delta_n^\stoc$. Moreover, the bound improves continuously as the approximation error of the harder nuisance $\delta_{\max}^\appr = \delta_\mu^\appr \lor \delta_\pi^\appr$ decreases from the hardest regime $\delta_{\max}^\appr \asymp 1$.

\begin{remark}[Inference]
Additionally assume that $\sigma_0 = \sqrt{\mathbb{E}[u^2 \varepsilon^2]} / \mathbb{E}[u^2] > 0$. Under the asymptotic normality conditions in \cref{coro:error-plm}, it is easy to verify that the following plug-in estimate of $\sigma_0$, denoted as $\hat{\sigma}$, satisfies $|\hat{\sigma} - \sigma_0|=o_{P}(1)$:
\begin{align*}
    \frac{\sqrt{n}(\hat{\theta} - \theta_0)}{\hat{\sigma}} \overset{d}{\to} \mathcal{N}(0, 1) \qquad \text{with} \qquad \hat{\sigma} = \frac{\left[\frac{1}{n} \sum_{i=1}^n (Y_i - \hat{\beta} \cdot T_i - \hat{\mu}(X_i))^2 (T_i - \hat{\pi}(X_i))^2 \right]^{1/2}}{\frac{1}{n} \sum_{i=1}^n (T_i - \hat{\pi}(X_i))^2}.
\end{align*} In this case, one can construct a confidence interval accordingly for inference. 
\end{remark}

\subsection{Structure-aware lower bound}
\label{sec:lb-plm}

In this subsection, we establish a no-gap result between the structure-agnostic setting and the worst-case structure-aware setting under the same approximation and stochastic error budgets. Specifically, the error rate \eqref{eq:intro-error} is attainable for any black-box class $\mathcal{G}$ by \cref{coro:error-plm}, yet remains unavoidable for a sparse linear class with an orthonormal dictionary. Thus, relative to the worst-case structure-aware model with the same budgets, structure agnosticism incurs no additional loss in estimation error rate. 
%In this subsection, we show that the error rate in \cref{coro:error-plm} is optimal when the learner class is characterized only through the approximation and stochastic error budgets. In fact, the same rate is already unavoidable for a sparse linear class with orthonormal bases. This further implies that learning the linear coefficient in the partial linear model by a generic black-box model is as easy as a sparse linear model from the perspective of statistical efficiency. 

We consider the simplified setting where $X \sim \nu_d := \mathrm{Unif}([0,1]^d)$, and let $\tfunc := \left\{f: \|f\|_{\infty, [0,1]^d} \le 3\right\}$. For a given pair of function classes $(\mathcal{F}_{\mu}, \mathcal{F}_{\pi})$ for $(\mu_0, \pi_0)$, define the associated family of distributions as 
\begin{align}
\label{eq:dgp-plm}
\begin{split}
    \mathcal{P}(\mathcal{F}_{\mu}, \mathcal{F}_\pi) = \Big\{ &(X, T, Y) \sim \mathbb{P}_{\beta, v, \mu,\pi}~\text{where}~ \mu \in \mathcal{F}_{\mu}, \pi\in \mathcal{F}_\pi, (\beta, v) \in [-3, 3] \times [1/3, 3]:\\
    &~~~~X\sim \nu_d, ~~ T = \pi(X) + u ~\text{with}~ \mathbb{E}[u|X]=0,\mathbb{E}[u^2|X] \equiv v, |T|\le 3,\\
    &~~~~Y = T \cdot \beta + \mu(X) + \varepsilon ~\text{with}~ \mathbb{E}[\varepsilon|X, T]=0,|Y| \le 3\Big\}.
\end{split}
\end{align} The target parameter is $\theta(\mathbb{P}) = \beta$ for $\mathbb{P} = \mathbb{P}_{\beta, v, \mu, \pi} \in \mathcal{P}(\mathcal{F}_\mu, \mathcal{F}_\pi)$. The restriction of the constant conditional variance function $\mathbb{E}[u^2|X]\equiv v$ is imposed to rule out the possibility that one can exploit additional structure in the treatment variance to attain a faster rate. 

The following proposition is analogous to Proposition 3.1 in \cite{gu2026optimally}: the approximation error product term follows similarly to the construction in \cite{balakrishnan2023fundamental, jin2025hard}, while the lower bound of $ s\log(p) / n$ uses the construction of sparse linear models in \cite{gu2026optimally}. 

\begin{proposition}
\label{prop:lb-plm}
    For any $\delta^\appr_\mu, \delta^\appr_\pi \in [0,1/3)$, and $p, s \in \mathbb{N}^+$ with $s\le p$ and $s[\log(p/s)+1] \le n \le p^{1/3}$, there exists a set of orthonormal bases $\{\phi_j(x)\}_{j=1}^p$ with $\mathbb{E}_{X\sim \nu_d}[\phi(X) \phi(X)^\top] = I_p$ for $\phi(X) = (\phi_1(X), \ldots, \phi_p(X))^\top \in \mathbb{R}^p$ such that the function classes defined as
    \begin{align*}
        \mathcal{F}_{\mu} = \mathcal{G}_{p, s} \cup \mathcal{B}(0, \delta^\appr_\mu), \qquad \text{and} \qquad \mathcal{F}_\pi = \mathcal{G}_{p, s} \cup \mathcal{B}(0, \delta^\appr_\pi),
    \end{align*} with $\mathcal{G}_{p, s} := \{g(x) = \sum_{j=1}^p \alpha_j \phi_j(x) \in \tfunc, \|\alpha\|_0 \le s \}$ and $\mathcal{B}(g, \delta) = \{f \in \mathcal{T}: \|f - g\|_{L_2(\nu_d)} \le \delta\}$ satisfies
    \begin{align}
    \label{eq:plm-lb}
         \inf_{\hat{\theta}} \sup_{\mathbb{P}\in \mathcal{P}(\mathcal{F}_\mu, \mathcal{F}_\pi)} \mathbb{E}_{\mathbb{P}^{2n}}\left[|\hat{\theta} - \theta(\mathbb{P})|\right] \ge \frac{1}{211} \left[n^{-1/2} + \delta_\mu^\appr \cdot \delta_\pi^\appr + \frac{s\log(p)}{n} \right].
    \end{align}  Here the randomness in $\mathbb{E}_{\mathbb{P}^{2n}}$ is $2n$ i.i.d. observations from $\mathbb{P}$, and $\hat{\theta}$ is the function of the $2n$ observations.
\end{proposition}

Because $0\in \mathcal{G}_{p,s}$, every function in $\mathcal{F}_\mu$ (resp., $\mathcal{F}_\pi$) can be approximated by functions in the class $\mathcal{G}_{p,s}$ with $L_2$ error at most $\delta_\mu^\appr$ (resp., $\delta_\pi^\appr$). Hence, \cref{prop:lb-plm} provides a lower bound in estimating the linear coefficient $\beta_0$ when both the outcome function and treatment function are approximately $s$-sparse linear models with misspecification error $\delta_\mu^\appr$ and $\delta_\pi^\appr$, respectively. Moreover, under the conditions of \cref{prop:lb-plm}, i.e., $s \le n \le p^{1/3}$, the stochastic error of the class $\mathcal{G}_{p,s}$ satisfies $\delta_n^\stoc \asymp \sqrt{s\log(p)/n}$; see for example, Lemma S1.1 in \cite{gu2026optimally}. Hence \cref{prop:lb-plm} establishes the structure-agnostic optimality: no estimator can guarantee a uniformly smaller error rate than \eqref{eq:intro-error} using only the three budgets.

Combining the lower bound in \cref{prop:lb-plm} and the upper bound in \cref{coro:error-plm} yields the promised no-gap claim. For any hypothesis class $\mathcal{G}$ with fixed approximation errors and stochastic error budgets $(\delta_\mu^\appr, \delta_\pi^\appr, \delta_n^\stoc)$, \cref{coro:error-plm} gives the estimation error
\begin{align*}
    \frac{1}{\sqrt{n}} + \delta^\appr_\mu \cdot \delta^\appr_\pi + (\delta_n^\stoc)^2,
\end{align*} while the above error rate is not improvable even when restricting $\mathcal{G}$ to be a sparse linear model $\mathcal{G}_{p, s}$ with $\delta_n^\stoc \asymp \sqrt{s\log(p)/n}$ as in \cref{prop:lb-plm}. Therefore, the optimal structure-agnostic rate matches the worst-case structure-aware rate under compatible error budgets. Hence, being agnostic to nuisance structure incurs no additional statistical cost. This extends the no-gap result of \cite{gu2026optimally}, which focuses on the regime $\delta_\pi^\appr \asymp 1$, to the general regime under which both $\delta_\mu^\appr$ and $\delta_\pi^\appr$ may vary.

%% file: 2-discussion.tex
\section{Transductive Adversarial Moment-calibrated Editing}
\label{sec:full-method}

\cref{sec:estimator} characterizes the best attainable performance for a generic black-box model and serves as a proof of concept for the general \our~framework. There, the same class $\mathcal{G}$ both produces the initial nuisance estimates and indexes the directions used to calibrate the edited debiasing weights. In this section, we decouple these two roles, allowing arbitrary initial black-box estimates to be calibrated against a separately chosen reference class $\mathcal{G}_0$.

The two cases in \cref{sec:estimator} then become two branches. The \texttt{O}-branch applies when the outcome regression function is easier to estimate and fine-tunes the treatment-side debiasing weights induced by the treatment estimate. The \texttt{R}-branch applies when the treatment regression function is easier to estimate and fine-tunes the outcome-side debiasing weights induced by the initial estimates.

\cref{subsec:plm-2} introduces the design principles and generic \our~algorithm. \cref{subsec:plm-oracle} establishes an oracle-type inequality for arbitrary initial estimates and a separately chosen calibration class.
\cref{subsec:plm-2-theory} specializes this result to two nuisance learners and compares \ourtwo~with DML. \cref{subsec:rate} then uses neural network learners to illustrate the under-smoothing principle for model selection and derive explicit rate comparisons. The resulting bounds show that \ourone~and \ourtwo~can improve on the standard DML guarantee when the nuisance difficulties are imbalanced, and that \ourtwo~can further improve on \ourone.

\subsection{Design principles and generic algorithm}
\label{subsec:plm-2}

%Leveraging the high-level principle behind the estimator in \cref{sec:estimator}, now we describe our {\it \textsf{T}ransductive \textsf{A}dversarial \textsf{M}oment-calibrated \textsf{E}diting (\textsf{TAME})} ($\our$) estimator that applies to two data splits $(\mathcal{D}_1, \mathcal{D}_2)$ and any initial black-box estimate $(\hat{\mu}, \hat{\beta}, \hat{\pi})$. The full procedure is presented in \cref{algo-tas}. 

Let $\mathcal{D}_1=\{(X_i,T_i,Y_i)\}_{i=1}^n$ and $\mathcal{D}_2=\{(X_i,T_i,Y_i)\}_{i=n+1}^{2n}$. \our~takes as input initial estimates $(\hat\mu,\hat\beta,\hat\pi)$ that are independent of $\mathcal{D}_1$, together with a reference class $\mathcal G_0$ used to index the adversarial moment constraints. The reference class need not be the class used to construct the initial estimates, although the oracle-type inequality requires it to contain the estimated easier nuisance.

Given a determination of which nuisance is easier, \our~use the corresponding branch $b \in\{\mathtt{O},\mathtt{R}\}$: the \bo-branch (resp., \br-branch) is intended for the case where the outcome (resp. treatment) regression function is estimated more accurately. Its central difference from DML is an additional post-prediction fine-tuning procedure: \our~retains the component associated with the easier nuisance but fine-tunes the score component (which we referred to as debiasing weights) induced by the harder nuisance, rather than using it directly as in DML. Specifically, the \texttt{O}-branch retains the outcome estimate $\hat\mu$ and fine-tunes the treatment-side weights $\widetilde{a}$, whereas the \texttt{R}-branch retains the treatment estimate $\hat\pi$ and fine-tunes the outcome-side weights $\tilde m$. The fine-tuning remains anchored at the corresponding black-box initial estimates while imposing adversarial moment constraints indexed by the reference class $\mathcal G_0$. Particularly, \our~is guided by a symbiosis of two criteria:
\begin{itemize}
    \item[(C1)] \emph{(Local editing)} In each branch, \our~starts from the same debiasing weights that DML would use for the harder nuisance. A centered ridge penalty discourages deviation from this initialization, so that the fine-tuning preserves good properties of the initial black-box estimate. In the two-learner specialization, this retains 
    \begin{align}
    \label{eq:abstract-error-1}
        (\text{approx. error for the easier nuisance}) \times (\text{$L_2$ error for the harder nuisance}).
    \end{align}
    \item[(C2)] \emph{(Transductive adversarial moment calibration)} Rather than leaving the harder-side weights at their initial value, \our~edits them directly on the inference sample to approximately satisfy the moment constraints indexed by $\mathcal{G}_0$. The weights are represented by an unrestricted vector in $\mathbb R^n$, optimized directly on the inference sample rather than represented by a fitted function of controlled complexity. Its statistical behavior is instead governed by adversarial moment constraints indexed by $\mathcal G_0$. In the \texttt{O}-branch, the edited weights use $\{(X_i, T_i)\}_{i=1}^n$, but not $\{Y_i\}_{i=1}^n$, from $\mathcal D_1$. In the \texttt{R}-branch, the moment constraints use the outcome $\{Y_i\}_{i=n+1}^{2n}$ in $\mathcal D_2$, while the edited values depend on $\mathcal{D}_1$ only through the covariates $\{X_i\}_{i=1}^n$. These constructions preserve the required conditional mean-zero relations. In the two-learner specialization, moment calibration replaces the DML propagation of nuisance errors, i.e., the product of the two nuisance $L_2$ errors, by the term in \eqref{eq:abstract-error-1} plus
    \begin{align*}
        (\text{$L_2$ error for the easier nuisance})^2.
    \end{align*} 
\end{itemize}

The three parts of the name describe the procedure directly: it is \emph{transductive}, mirroring transductive learning, because it directly optimizes observation-specific debiasing weights on the inference sample rather than learning a function intended to generalize to unseen observations, \emph{adversarially moment-calibrated} because it enforces worst-case moments over the entire class $\partial \mathcal{G}_0$, and an \emph{editing} procedure because the optimized weights remain close to the initial black-box-induced weights.

\cref{algo-tas} can be run with any initial estimates that are independent of $\mathcal D_1$. For the oracle-type inequality in \cref{subsec:plm-oracle}, the reference class is required to contain the estimate of the easier nuisance: $\widehat\mu\in\mathcal G_0$ in the \bo-branch and $\widehat\pi\in\mathcal G_0$ in the \br-branch. When the initial estimates are obtained from \eqref{eq:est1-step1} and $\mathcal G_0=\mathcal G$, the \bo- and \br-branches reduce to $\widehat\theta_\mu$ and $\widehat\theta_\pi$, respectively. Thus, \ourone~is a special case of the generic algorithm.

\begin{algorithm}[!t]
\caption{$\our$ for partial linear model}
\label{algo-tas}
\begin{algorithmic}[1]
\State \textbf{Input:} Data split: $\mathcal{D}_1 = \{(X_i, T_i, Y_i)\}_{i=1}^n$, $\mathcal{D}_2 = \{(X_i, T_i, Y_i)\}_{i=n+1}^{2n}$; $\hat{\mu}, \hat{\beta}, \hat{\pi}$ independent of $\mathcal{D}_1$, debiasing reference class $\mathcal{G}_0$; branch choice $b\in \{\bo, \br\}$, hyper-parameter $\lambda > 0$.
\If{$b = \bo$} \Comment{\mypurple{Branch $\bo$: the outcome part is easier}}
    \State Initialize debiasing weights $a\in \mathbb{R}^n$ with $a_i \gets \tilde{a}_i := 
    \left[\left\{\frac1n\sum_{\ell=1}^n T_\ell\{T_\ell-\hat\pi(X_\ell)\}\right\}\vee
    \frac1{2c}\right]^{-1} \{T_i-\widehat\pi(X_i)\}$.
    \State Fine-tune weights $a$ by adversarial estimation on $(\mathcal{G}_0, \mathcal{D}_1)$ and get fine-tuned weights $\hat{a}$:
    \begin{align*}
        \hat{a} \in \argmin_{a\in \mathbb{R}^n} \frac{\lambda}{n} \sum_{i=1}^n (a_i - \tilde{a}_i)^2 + \sup_{\beta\in \mathbb{R}, f \in \partial \mathcal{G}_0} \Phi_\mu(f, \beta, a),
    \end{align*} ~~~~~where $\Phi_\mu(f, \beta, a) := \left| \frac{1}{n}\sum_{i=1}^n (\beta \cdot T_i + f(X_i)) \cdot a_i - \beta\right| - \frac{1}{n} \sum_{i=1}^n f^2(X_i)$.
    \State Assign $\hat{\theta} \gets \hat{\theta}_{\bo} = \frac{1}{n} \sum_{i=1}^n (Y_i - \hat{\mu}(X_i)) \cdot \hat{a}_i$. 
\Else \Comment{\mypurple{Branch $\br$: the treatment part is easier}}
    \State Initialize debiasing weights $m\in \mathbb{R}^n$ with $m_i \gets \tilde{m}_i := \hat{\beta} \cdot \hat{\pi}(X_i) + \hat{\mu}(X_i)$.
    \State Fine-tune weights $m$ by adversarial estimation on $(\mathcal{G}_0, \mathcal{D}_1, \mathcal{D}_2)$ and get fine-tuned weights $\hat{m}$:
    \begin{align*}
        \hat{m} \in \argmin_{m\in \mathbb{R}^n} \frac{\lambda}{n} \sum_{i=1}^n (m_i - \tilde{m}_i)^2 + \sup_{f\in \partial \mathcal{G}_0} \Phi_\pi(f, m),
    \end{align*}
    ~~~~~where $\Phi_\pi(f, m) := \left|\frac{1}{n} \sum_{i=n+1}^{2n} Y_i f(X_i) - \frac{1}{n} \sum_{i=1}^n m_i f(X_i) \right| - \frac{1}{n} \sum_{i=1}^n f^2(X_i)$.
    \State Assign $\hat{\theta} \gets \hat{\theta}_{\br} = [(2\const)^{-1} \lor \frac{1}{n} \sum_{i=1}^n (T_i - \hat{\pi}(X_i))^2]^{-1} [\frac{1}{n} \sum_{i=1}^n (Y_i - \hat{m}_i)(T_i - \hat{\pi}(X_i))]$. 
\EndIf
\State \textbf{Output:} The final estimate $\hat{\theta}$. 
\end{algorithmic}
\end{algorithm}

\subsection{A generic oracle-type inequality}
\label{subsec:plm-oracle}

\cref{thm:plm2} gives an oracle-type inequality for \cref{algo-tas} with arbitrary initial estimates $(\hat{\mu}, \hat{\beta}, \hat{\pi})$ and a separately chosen reference class $\mathcal G_0$. The result separates the errors of the initial estimates from the approximation and stochastic errors associated with $\mathcal{G}_0$, thereby making explicit the contributions of local editing and adversarial moment calibration. For concision, we state the result under an oracle choice of $\lambda$; see the full error with explicit dependence on $\lambda>0$ in the proof. 

\begin{condition}[Reference class and initial estimates]\label{cond:plm-bb} The following holds:
\begin{itemize}
\item[(a)] As in \cref{def:hcs}, $\mathcal{G}_0 \in \mathcal{HCS}(\delta_n^\stoc; n, \nu_x)$, where $\nu_x$ is the distribution of $X$.
\item[(b)] $\hat{\mu}, \hat{\pi}, \hat{\beta}$ are uniformly bounded, i.e., $\|\hat{\mu}\|_\infty\lor \|\hat{\pi}\|_\infty\lor |\hat{\beta}| \le \const$, and are independent of $\mathcal{D}_1$.
\end{itemize}
\end{condition}

\begin{theorem} \label{thm:plm2} Assume \cref{cond:reg-plm} and \cref{cond:plm-bb}. Recall $\hat{\varphi}_{\mathsf{PLM}}$ defined in \eqref{eq:psi-plm}. Then there exists a constant $\tilde{C}=\poly(\const)$ such that if $\delta_n^\stoc \le 1/\tilde{C}$, then for any $t \in [1,n/\tilde{C}]$, with probability at least $1-\left[2e^{-t}+\exp(-n(\delta_n^\stoc)^2)\right]$, the following holds:
\begin{itemize}
    \item[$(\bo)$] If $\hat{\mu} \in \mathcal{G}_0$ and $\|\hat{\pi} - \pi_0\|_2 \le \delta_\pi$ for $\delta_\pi \in \mathbb{R}^+$, then \cref{algo-tas} with branch $b=\bo$ and $\lambda \asymp (\delta_n^\stoc / \delta_\pi)^2$ returns $\hat{\theta}$ satisfying
    \begin{align}
    \label{eq:plm-branch-o}
    \begin{split}
        \tilde{C}^{-1} |\hat{\theta} - \theta_0 - \hat{\varphi}_{\mathsf{PLM}}| &\le \inf_{g\in \mathcal{G}_0} \Bigg[\myblue{\|g - \mu_0\|_2 \cdot \delta_\pi} + \myred{\|g - \hat{\mu}\|_2^2 + [\delta_n^\stoc]^2} \\
        &\qquad \qquad + \mygray{\sqrt{\frac{t}{n}} \left\{ \delta_\pi + \|g - \mu_0\|_2 + (1+\lambda) \sqrt{\frac{t}{n}} \right\}}\Bigg].
    \end{split}
    \end{align}
    \item[$(\br)$] If $\hat{\pi} \in \mathcal{G}_0$ and $\|\hat{\mu} - \mu_0\|_2 + \|\hat{\pi} - \pi_0\|_2 + |\hat{\beta} - \beta_0| \le \delta_{all}$ with $\delta_{all} \in \mathbb{R}^+$, then \cref{algo-tas} with branch $b=\br$ and $\lambda \asymp (\delta_n^\stoc / \delta_{all})^2$ returns $\hat{\theta}$ satisfying
    \begin{align}
    \label{eq:plm-branch-r}
        \begin{split}
        \tilde{C}^{-1} |\hat{\theta} - \theta_0 - \hat{\varphi}_{\mathsf{PLM}}| &\le \inf_{g\in \mathcal{G}_0} \Bigg[\myblue{\|g - \pi_0\|_2 \cdot \delta_{all}} + \myred{\|g - \hat{\pi}\|_2^2 + [\delta_n^\stoc + \|\hat{\pi} - \pi_0\|_2]^2} \\
        &\qquad \qquad + \mygray{\sqrt{\frac{t}{n}} \left\{ \delta_{all} + \|g - \pi_0\|_2 + (1+\lambda) \sqrt{\frac{t}{n}} \right\}}\Bigg].
    \end{split}    \end{align}
\end{itemize}
\end{theorem}

For both branches, the first \myblue{blue} term is the direct result of the criterion [C1]: it multiplies the approximation error of the reference class by the $L_2$ error of the initial estimate on the opposite side. The subsequent \myred{red} quadratic terms arise from the adversarial moment calibration in criterion [C2]. The last \mygray{gray} term is of order $o(n^{-1/2})$ when consistent estimations of $\mu_0$ and $\pi_0$ are attained and $\lambda = o(\sqrt{n})$. In the \ourtwo~specialization below, the tuning hyperparameter $\lambda \le 1$, so this additional requirement is automatic. Their roles will be more explicit in the next subsection, when we use two black-box models to estimate $\mu_0$ and $\pi_0$ separately. 

The error bound for one black-box learner in \cref{coro:error-plm} is a special case of the above theorem by picking $g$ in \eqref{eq:plm-branch-o} (resp. \eqref{eq:plm-branch-r}) as the near-best approximation of $\mu_0$ (resp. $\pi_0$) in the class $\mathcal{G} = \mathcal{G}_0$.

\subsection{\ourtwo: a two-learner specialization}
\label{subsec:plm-2-theory}

When separate black-box models $\mathcal{G}_\mu$ and $\mathcal{G}_\pi$ are used to estimate $\mu_0$ and $\pi_0$, respectively, a natural specialization of our $\our$ method, which we refer to as $\ourtwo$ estimator, first obtains $(\hat{\mu}, \hat{\beta})$ and $\hat{\pi}$ by (joint) least squares. It then uses the branch $\bo$ (resp. $\br$) with reference hypothesis class $\mathcal{G}_0\gets \mathcal{G}_\mu$ (resp. $\mathcal{G}_0 \gets \mathcal{G}_\pi$) when the estimation of the outcome (resp., treatment) function is more accurate. 

In this subsection, we derive the resulting rate and compare it with the standard DML guarantee. \ourtwo~matches the DML guarantee up to constants and improves it when the nuisance difficulties are imbalanced and the easier learner is appropriately under-smoothed.

We first formalize the approximation and stochastic error budgets of the two learners.

\begin{condition}\label{cond:plm-two-learner}
Let $\nu_x$ be the distribution of $X$. Assume $\mathcal{G}_\mu \in \mathcal{HCS}(\delta_\mu^\stoc; n, \nu_x)$, $\mathcal{G}_\pi \in \mathcal{HCS}(\delta_\pi^\stoc; n, \nu_x)$. Define 
\begin{align*}
    \delta_\mu^\appr := \inf_{g\in \mathcal{G}_\mu} \|g - \mu_0\|_{L_2(\nu_x)} \qquad \text{and} \qquad \delta_\pi^\appr := \inf_{g\in \mathcal{G}_\pi} \|g - \pi_0\|_{L_2(\nu_x)}.
\end{align*}
\end{condition}

Given the split $\mathcal{D}_1 \cup \mathcal{D}_2$, both our $\ourtwo$ estimator and DML estimator use the following initial estimates obtained from $\mathcal{D}_2$:
\begin{align}
\label{eq:est2-step1}
    (\hat{\beta}, \hat{\mu}) \in \argmin_{\substack{\beta \in [-\const, \const]\\ g \in \mathcal{G}_\mu}} \frac{1}{n} \sum_{i=n+1}^{2n} \left(Y_i - \beta \cdot T_i - g(X_i)\right)^2, \qquad 
    \hat{\pi} \in \argmin_{g\in \mathcal{G}_\pi} \frac{1}{n} \sum_{i=n+1}^{2n} \left(T_i - g(X_i)\right)^2.
\end{align} It follows from the standard empirical process argument that, with high probability,
\begin{align*}
     \|\hat{\mu} - \mu_0\|_2 + |\hat{\beta} - \beta_0| \lesssim \delta_\mu^\appr + \delta_\mu^\stoc =: \delta_{\mu} \qquad \text{and} \qquad \|\hat{\pi} - \pi_0\|_2 \lesssim \delta_\pi^\appr + \delta_\pi^\stoc =: \delta_{\pi}.
\end{align*} Using the initial estimates in \eqref{eq:est2-step1}, the standard DML estimator $\hat{\theta}_{\dml}$ is:
\begin{align}
\label{eq:est-dml}
    \hat{\theta}_{\dml} = \left\{\frac{1}{n} \sum_{l=1}^n T_l(T_l - \hat{\pi}(X_l)) \right\}^{-1}  \frac{1}{n} \sum_{i=1}^n \left(Y_i - \hat{\mu}(X_i)\right)(T_i - \hat{\pi}(X_i)).
\end{align}
Using the upper error budgets, define \ourtwo~by the following branch and tuning choices:
\begin{align}
\label{eq:est-taff2}
    \hat{\theta}_{\ourtwo} = \begin{cases}
        \hat{\theta} \text{ from \cref{algo-tas} with } b=\bo, \mathcal{G}_0 = \mathcal{G}_\mu, \lambda = (\delta_\mu^\stoc / \delta_\pi)^2 & \qquad \text{if }\delta_\mu \le \delta_\pi \\
        \hat{\theta} \text{ from \cref{algo-tas} with } b=\br, \mathcal{G}_0 = \mathcal{G}_\pi, \lambda = (\delta_\pi^\stoc / \delta_\mu)^2 & \qquad \text{if }\delta_\pi < \delta_\mu \\
    \end{cases}.
\end{align}

The following corollary compares the nuisance contributed remainders of $\hat{\theta}_{\dml}$ and $\hat{\theta}_{\ourtwo}$.
%The following corollary gives the error bounds of $\hat{\theta}_{\dml}$ and $\hat{\theta}_{\ourtwo}$, the error bound for $\hat{\theta}_{\ourtwo}$ by realizing the error bounds in \cref{thm:plm2} where \myblue{blue}, \myred{red}, and \mygray{gray} corresponds the components with same semantic meaning. 
\begin{corollary} \label{coro:plm-2}
    Under \cref{cond:reg-plm} and \cref{cond:plm-two-learner}, suppose further $\delta_\mu + \delta_\pi = o(1)$. Then there exists $R_{\dml}, R_{\ourtwo}$, satisfying $R_{\dml} \lor R_{\ourtwo} = o(n^{-1/2})$, such that, with high probability    
    \begin{align}
    \label{eq:error-two-learner}
    \begin{split}
        |\hat{\theta}_{\dml} - \theta_0 - \hat{\varphi}_{\mathsf{PLM}}| &\lesssim \delta_\mu \cdot \delta_\pi + \mygray{R_{\dml}}, \\
        |\hat{\theta}_{\ourtwo} - \theta_0 - \hat{\varphi}_{\mathsf{PLM}}| &\lesssim \left[\myblue{\delta^\appr_\mu \cdot \delta_\pi} + \myred{(\delta_\mu)^2} \right]\land \left[\myblue{\delta^\appr_\pi \cdot \delta_\mu} + \myred{(\delta_\pi)^2} \right] + \mygray{R_{\ourtwo}}.
    \end{split}
    \end{align}
\end{corollary}

To interpret \eqref{eq:error-two-learner}, suppose first that $\delta_\mu \le \delta_\pi$, so the outcome regression is easier and \ourtwo~uses the \bo-branch. DML has nuisance remainder $\delta_\mu \cdot \delta_\pi$, whereas \ourtwo~has $\delta_\mu^{\mathrm{appr}}\delta_\pi+\delta_\mu^2$. Since
\[
\frac{\text{\ourtwo~nuisance error}}{\text{\dml~nuisance error}}=\frac{\delta_\mu^{\mathrm{appr}}\delta_\pi+\delta_\mu^2}{\delta_\mu\delta_\pi}=\frac{\delta_\mu^{\mathrm{appr}}}{\delta_\mu}+\frac{\delta_\mu}{\delta_\pi},
\] and $\delta_\mu^{\mathrm{appr}}\le\delta_\mu\le\delta_\pi$, \ourtwo~is no worse than DML up to constants. 
Furthermore, if $\delta_\mu\ll\delta_\pi$ and $\delta_\mu^\appr\ll\delta_\mu^\stoc$, then the \ourtwo~nuisance remainder is $o(\delta_\mu\delta_\pi)$. The second condition $\delta_\mu^\appr\ll\delta_\mu^\stoc$ is precisely the under-smoothing requirement for the easier outcome part. The treatment-easier case is symmetric. If, in addition, $\delta_\mu\delta_\pi\gg n^{-1/2}$, this reduction yields a strictly faster overall estimation rate than DML. \cref{subsec:rate} makes the comparison explicit for neural network learners.

\subsection{Application to the neural network learners via under-smoothing}
\label{subsec:rate}

%In this subsection, we use a neural network class as an instantiation of the black-box learner to illustrate how to optimally use $\ourone$, which applies our method to one black-box learner as in \cref{sec:estimator}, and $\ourtwo$ in the previous subsection via the idea of under-smoothing \citep{rice1986convergence}, and compare it with the best rate the DML procedure gives. Let $\delta_{\dml}$, $\delta_{\ourone}$, and $\delta_{\ourtwo}$ be the estimation error for $\theta_0$ using different methods with optimal model selection. We will show that
In this subsection, we instantiate the black-box learners by neural networks and use them to illustrate the under-smoothing principle \citep{rice1986convergence}. Let $r_{\dml}$, $r_{\ourone}$, and $r_{\ourtwo}$ denote the optimally tuned convergence rate for estimating $\theta_0$, with poly-logarithmic factors suppressed. We show that
\begin{align} \label{eq:nn-comparison}
    r_{\dml} \gtrsim r_{\ourone} \gtrsim r_{\ourtwo}.
\end{align} The conditions under which these comparisons are strict are given after \cref{coro:nn-plm}. We first introduce the neural network (NN) class and nuisance function class. Then we present the estimator with optimal model selection and the resulting rates.

\myparagraph{Neural network class. } We consider the fully connected deep neural network with ReLU activation $\mathrm{ReLU}(\cdot) = \max\{0, \cdot\}$. For positive integers $L$ and $N$, a \emph{deep ReLU network with depth $L$ width $N$} admits the form of
\begin{align}
\label{eq:nn-architecture}
    g(x) = T_{L+1} \circ \overline{\mathrm{ReLU}}_L \circ T_L \circ \bar{\mathrm{ReLU}}_{L-1} \circ \cdots \circ T_2 \circ \overline{\mathrm{ReLU}}_1 \circ T_1(x).
\end{align} Here $T_{l}(z) = W_l z + b_l: \mathbb{R}^{d_{l-1}} \to \mathbb{R}^{d_{l}}$ with $l\in [L+1]$ is an affine map with weight matrix $W_l \in \mathbb{R}^{d_{l}\times d_{l-1}}$ and bias vector $b_{l} \in \mathbb{R}^{d_{l}}$, where $(d_0,d_1\ldots, d_L, d_{L+1}) = (d, N, \ldots, N, 1)$, and $\overline{\mathrm{ReLU}}_l: \mathbb{R}^{d_l} \to \mathbb{R}^{d_l}$ applies the ReLU activation ${\mathrm{ReLU}}(\cdot)$ to each entry of a $d_l$-dimensional vector for each $l\in [L]$. We define the neural network class as follows. 

\begin{definition}[Deep ReLU network class]
    Define the family of deep ReLU networks with a $d$-dimensional input, depth $L$, width $N$, and truncation level $B$ by 
        $$\mathcal{F}_{\mathtt{nn}}(d, L, N, B) = \{\tilde{g}(\cdot) = \mathrm{Tc}_B(g(\cdot)): g(\cdot) \text{ in } \eqref{eq:nn-architecture}\},$$ 
        where $\mathrm{Tc}_B: \mathbb{R} \to \mathbb{R}$ is the truncation operator defined as $\mathrm{Tc}_B(z) = {\min}\{|z|, B\} \cdot \mathrm{sign}(z)$.
\end{definition}

\myparagraph{Nuisance function class.} Here we introduce the concept of the hierarchical composition model \citep{bauer2019deep, schmidt2020nonparametric, kohler2021rate}, a low-dimensional function class that neural networks can efficiently estimate. 

\begin{definition}[Hierarchical composition model $\mathcal{F}_{\mathtt{HCM}}(d, \ell, \mathcal{O}, C)$]
\label{hcm}
    Let $\ell, d \in \mathbb{N}^+$, $C>0$, and let $\mathcal{O} \subset [1,\infty) \times \mathbb{N}^+$ be a non-empty finite set. Define $\mathcal{F}_{\mathtt{HCM}}(d, \ell, \mathcal{O}, C)$ \citep{kohler2021rate} recursively as follows. Let $\mathcal{F}_{\mathtt{HCM}}(d, 0,\mathcal{O}, C)=\{h(x)=x_j, j\in [d]\}$, and for each $\ell \ge 1$, 
    \begin{align*}
    \mathcal{F}_{\mathtt{HCM}}(d, \ell,\mathcal{O}, C) = \big\{&h: \mathbb{R}^d \to \mathbb{R}: h(x) = g(f_1(x),...,f_r(x))\text{, where} \\
    &~~~~~ g\in \mathcal{F}_{\mathtt{HS}}(r, \beta, C) \text{ with } (\beta, r)\in \mathcal{O} \text{ and } f_i \in \mathcal{F}_{\mathtt{HCM}}(d, \ell-1,\mathcal{O}, C)\big\}. 
\end{align*} Here $\mathcal{F}_{\mathtt{HS}}(d, \beta, C)$ denotes the set of all the $d$-variate $(\beta, C)$-H{\"o}lder smooth functions in \cite{gu2024causality}.
\end{definition}

The class $\mathcal{F}_{\mathtt{HCM}}(d, \ell,\mathcal{O}, C)$ consists of finite compositions of $r$-variate functions with H{\"o}lder smoothness $\beta$, where $(\beta, r) \in \mathcal{O}$. Define the smallest dimension-adjusted smoothness \citep{fan2024factor} $\gamma^\star = \min_{(\beta, r)\in \mathcal{O}} (\beta/r)$. For example, if $f(x) = f_1(x_1) + f_2(f_3(x_2, x_3), f_4(x_1, x_4), x_5)$ and all functions have smoothness $\beta=2$, then the hardest component is $f_2$ with $(\beta, r) = (2, 3)$, resulting in $\gamma^\star = 2/3$. 

If $f_0$ lies in $\mathcal{F}_{\mathtt{HCM}}(d, \ell,\mathcal{O}, C)$, then the approximation error \citep{fan2024noise} and stochastic error for a neural network class $\mathcal{F}_{\mathtt{nn}}(d, L, N, B)$ can be represented by its statistical complexity $S=(NL)^2$ as 
\begin{align}
\label{eq:error-nn}
    \delta^\appr \lesssim \poly(\log(S)) \cdot S^{-\gamma^\star} \qquad \text{and} \qquad \delta^\stoc \lesssim \poly(\log(Sn)) \cdot \sqrt{\frac{S}{n}}.
\end{align} See, for example, proof of Theorem 2.1 in \cite{gu2024causality}. Hence, a neural network estimator $\hat{f}$ can optimally estimate the target regression function $f_0 \in \mathcal{F}_{\mathtt{HCM}}(d, \ell,\mathcal{O}, C)$ by choosing $S \asymp n^{\frac{1}{(2\gamma^\star+1)}} =: \mathbb{S}(\gamma^\star)$ that balances $\delta^\appr \asymp \delta^\stoc$, resulting in a $L_2$ estimation error $\|\hat{f} - f_0\|_2 \lesssim \poly(\log(n)) \cdot n^{-\frac{\gamma^\star}{2\gamma^\star+1}}$, which matches the minimax optimal $L_2$ estimation risk over $\mathcal{F}_{\mathtt{HCM}}(d, \ell, \mathcal{O}, C)$.

\myparagraph{Model selection for optimal convergence rate.} We next impose regularity conditions under which the approximation error and stochastic error budgets in \eqref{eq:error-nn} apply to both nuisance functions. \cref{cond:nn-plm} specifies the HCM structure, boundedness on the data generating process, and the neural network architectures used by \dml, \ourone, and \ourtwo~estimators. %We first impose a condition on the data-generating process and regularity conditions on the function class $\mathcal{G}, \mathcal{G}_\mu, \mathcal{G}_\pi$ that may be used for \dml, \ourone, or \ourtwo~ estimator: Here (a) specifies the low-dimensional structure $\mu_0$ and $\pi_0$ have and assume that both functions can be estimated at a polynomial rate; (b) is a regularity condition similar to \cref{cond:reg-plm} but additionally assumes the covariate $X$ is also uniformly bounded, which is used for deriving approximation error rates for neural networks. We also require the NN class used to have at least $\poly\log(n)$ growing depth and width, which is also of regularity nature. 

\begin{condition} \label{cond:nn-plm} Recall the partial linear model in \eqref{eq:intro:model}, there exists a constant $c_1 > 1$ such that:
\begin{itemize}
\item[(a)] Let $\mathcal{O}_\mu$ and $\mathcal{O}_\pi$ be two sets of smoothness-dimension indices, $\ell \in \mathbb{N}^+$, $C_0 \in \mathbb{R}^+$ such that $$\max_{(\beta, r) \in \mathcal{O}_\mu \cup \mathcal{O}_\pi} \allowbreak \{\beta \lor r\} \lor \ell \lor C_0 \lor d \le c_1,$$ $f_0 \in \mathcal{F}_{\mathtt{HCM}}(d, \ell, \mathcal{O}_f, C_0)$ with $\gamma_f^\star = \min_{(\beta, r)\in \mathcal{O}_f} (\beta/r)$ for $f\in \{\mu, \pi\}$. We assume $\gamma_\mu^\star \land \gamma_\pi^\star > 0$.  
\item[(b)] The nuisance functions and variables are uniformly bounded: $\|X\|_\infty \le c_1$, $|\varepsilon| \lor |u| \le c_1$, $|\beta_0| \lor \|\mu_0\|_\infty \lor \|\pi_0\|_\infty \le c_1$, and $\mathbb{E}[u^2] \ge c_1^{-1}$.
\item[(c)] We use $\mathcal{G}_1 = \mathcal{F}_{\mathtt{nn}}(d, L_1, N_1, B)$, $\mathcal{G}_\mu = \mathcal{F}_{\mathtt{nn}}(d, L_\mu, N_\mu, B)$ and $\mathcal{G}_\pi = \mathcal{F}_{\mathtt{nn}}(d, L_\pi, N_\pi, B)$ with same $B=c_1$ but different depth and width parameters $(L_q,N_q)$ for $q\in \{1, \mu, \pi\}$ satisfying $L_q \land N_q \ge c_2 \log(n)$ for some constant $c_2$ dependent on $c_1$. We also denote $S_q = (N_q L_q)^2$ for $c\in \{1, \mu, \pi\}$. 
\end{itemize}
\end{condition} 

The following corollary calculates the convergence rate for different methods under optimal model selection by plugging the choice of $S$ into the approximation error and stochastic error budgets \eqref{eq:error-nn}: the inequality \eqref{eq:nn-rate} establishes the rate ordering in \eqref{eq:nn-comparison}. 

\begin{corollary}
\label{coro:nn-plm}
Under \cref{cond:nn-plm}, recall $\mathbb{S}(\gamma^\star) = n^{\frac{1}{(2\gamma^\star+1)}}$, consider the following three estimators,
\begin{itemize}
    \item[(a)] $\hat{\theta}_{\dml, \textsf{NN}} \gets \hat{\theta}_{\dml}$ in \eqref{eq:est-dml} with $\mathcal{G}_\mu, \mathcal{G}_\pi$ satisfying $S_\mu \asymp \mathbb{S}(\gamma_\mu^\star)$ and $S_\pi \asymp \mathbb{S}(\gamma_\pi^\star)$.
    \item[(b)] $\hat{\theta}_{\ourone, \textsf{NN}} \gets \hat{\theta}$ in \eqref{eq:est1-est} with $\mathcal{G}=\mathcal{G}_1$ satisfying $S_1 \asymp n^{\frac{1}{\gamma^\star_\mu+\gamma^\star_\pi+1}}$ and the optimal choice of $\lambda$ in \cref{thm:main-plm}. Here $S_1$ balances the polynomial upper bounds of $\delta_\mu^\appr \cdot \delta_\pi^\appr$ and $(\delta_n^\stoc)^2$ as $S_1^{-(\gamma^\star_\mu+\gamma^\star_\pi)} \asymp S_1 / n$.
    \item[(c)] $\hat{\theta}_{\ourtwo, \textsf{NN}} \gets \hat{\theta}_{\ourtwo}$ in \eqref{eq:est-taff2} with $\mathcal{G}_\mu$ and $\mathcal{G}_\pi$ satisfying the following hyper-parameter choice 
    \begin{align}
        \begin{cases}
            {S_\pi \asymp \mathbb{S}(\gamma_\pi^\star)}, {S_\mu \asymp n^{\frac{\gamma_\pi^\star+1}{2\gamma_\pi^\star+1}\cdot \frac{1}{\gamma_\mu^\star+1}}} \gg \mathbb{S}(\gamma_\mu^\star) \qquad & \text{if } \gamma_\mu^\star > \gamma_\pi^\star \\
            {S_\mu \asymp \mathbb{S}(\gamma_\mu^\star)}, S_\pi \asymp \mathbb{S}(\gamma_\pi^\star) \qquad & \text{if } \gamma_\pi^\star = \gamma_\mu^\star \\
            {S_\mu \asymp \mathbb{S}(\gamma_\mu^\star)}, {S_\pi \asymp n^{\frac{\gamma_\mu^\star+1}{2\gamma_\mu^\star+1}\cdot \frac{1}{\gamma_\pi^\star+1}}}\gg \mathbb{S}(\gamma_\pi^\star) \qquad & \text{if } \gamma_\pi^\star > \gamma_\mu^\star
        \end{cases}. \label{eq:model-selection-2}
    \end{align} 
\end{itemize} For any $\mathsf{M} \in \{\dml, \ourone, \ourtwo\}$, $\hat{\theta}_{\mathsf{M}, \mathsf{NN}}$ satisfies $|\hat{\theta}_{\mathsf{M}, \mathsf{NN}} - \theta_0| \lesssim n^{-1/2} + \poly(\log(n)) \cdot n^{-\alpha_{\mathsf{M}}}$ with high probability, where
    \begin{align}
    \label{eq:nn-rate}
         \underbrace{\frac{\gamma_\mu^\star}{2\gamma_\mu^\star+1} + \frac{\gamma_\pi^\star}{2\gamma_\pi^\star+1}}_{=:\alpha_{\dml}} \le \underbrace{\frac{\gamma_\mu^\star+\gamma_\pi^\star}{\gamma_\mu^\star+\gamma_\pi^\star+1}}_{=:\alpha_{\ourone}} \le \underbrace{\frac{s^\star}{s^\star + \{\gamma^\star_\mu\land \gamma^\star_\pi\}+1}}_{=:\alpha_{\ourtwo}}
    \end{align} and where $s^\star = 2\gamma_\mu^\star \gamma_\pi^\star + \gamma_\mu^\star + \gamma_\pi^\star$, with both equalities in \eqref{eq:nn-rate} holding if and only if $\gamma_\mu^\star = \gamma_\pi^\star$.
\end{corollary}

%The claim \eqref{eq:nn-comparison} at the beginning of this subsection, i.e., the strict improvement of \ourone~over \dml~and the strict improvement of \ourtwo~over \ourone~follows directly from the inequality \eqref{eq:nn-rate}. We comment on the model selection criterion for different methods below. 
%\eqref{eq:nn-comparison} establishes the rate ordering in \eqref{eq:nn-rate}. Under \cref{cond:nn-plm}, both exponent inequalities are strict whenever $\gamma_\mu^\star \neq \gamma_\pi^\star$. These exponent gains further yield strict improvements in the overall estimation error rate whenever the nonparametric term dominates. 

\begin{remark}[Optimal model selection for DML] For DML estimator, the nuisance errors enter in the final error as $\delta_\mu \cdot \delta_\pi$ in \cref{coro:plm-2}, which means the model selections for $\mathcal{G}_\mu$ and $\mathcal{G}_\pi$ should match that for optimal estimations of $\mu_0$ and $\pi_0$, i.e., $\delta_\mu^\appr \asymp \delta_\mu^\stoc$ and $\delta_\pi^\appr \asymp \delta_\pi^\stoc$. 
\end{remark}
\begin{remark}[Optimal model selection for \our]
For \ourone, the nuisance error budgets enter in the final error as $\delta^\appr_\mu \cdot \delta^\appr_\pi + (\delta_n^\stoc)^2$, which suggests the corresponding model selection rule $\delta^\appr_\mu \cdot \delta^\appr_\pi \asymp (\delta_n^\stoc)^2$. This choice under-smooths the easier nuisance, but over-smooths the harder nuisance. \ourtwo~avoids the drawbacks introduced by the latter restriction by using separate neural networks. To see this, when $\gamma_\mu^\star > \gamma_\pi^\star$, i.e., the outcome is easier to estimate,  \cref{coro:plm-2} gives the oracle-type error bound $(\delta_\pi^\appr + \delta_\pi^\stoc) \cdot \delta_\mu^\appr + \delta_\mu^2$, this suggests doing standard bias-variance tradeoff on the harder $\pi_0$ part $\delta_\pi^\appr \asymp \delta_\pi^\stoc$ and further under-smoothing on easier $\mu_0$ part $\delta_\mu^\appr \cdot \delta_\pi^\star \asymp (\delta_\mu^\stoc)^2$ on top of the balanced $\delta_\pi^\star := \delta_\pi^\appr \asymp \delta_\pi^\stoc$ as in the first line of \eqref{eq:model-selection-2}. The mirror case $\gamma_\pi^\star \ge \gamma_\mu^\star$ follows similarly.
\end{remark}
\begin{remark}[The case where $\alpha_{\mathsf{M}}$ coincides] Under \cref{cond:nn-plm}, both inequalities in \eqref{eq:nn-rate} are strict unless $\gamma_\mu^\star = \gamma_\pi^\star$. Thus, at the nonparametric rate level, \ourone~strictly improves on \dml, and \ourtwo~strictly improves on \ourone~when the nuisance difficulties are imbalanced. These exponent gains further yield strict improvements in the overall estimation error upper bound whenever the nonparametric term dominates $n^{-1/2}$. 
\end{remark}

\section{Extensions beyond the partial linear model}
\label{sec:extension}

This section extends \our~beyond the partial linear model. We discuss the extension of \texttt{O}-branch to a broad class of linear functional estimation problems in \cref{subsec:o-branch}. The extension of \texttt{R}-branch for estimating average treatment effect (ATE) is discussed in \cref{subsec:r-branch}. With a slight abuse of notation, we adopt $\mu$, $\theta$ with different semantic meanings than those for the partial linear model before.  

\subsection{The \texttt{O}-branch for general linear functionals}
\label{subsec:o-branch}

\myparagraph{Setup.} Consider the following general model for $(X, T, Y) \in \mathbb{R}^{d} \times \mathbb{R} \times \mathbb{R}$:
\begin{align}
\label{eq:model-general}
    Y = \mu_0(X, T) + \varepsilon \qquad \text{with} \qquad \mu_0 \in \mathcal{H}_\mu, ~\mathbb{E}[\varepsilon|X, T] = 0, 
\end{align} where $\mathcal{H}_{\mu}$ is a Hilbert space equipped with the inner product $\langle h_1, h_2\rangle = \mathbb{E}_{\nu_z}[h_1(Z) h_2(Z)]$, and where $\nu_z$ is the distribution of $Z = (X, T) \in \mathbb{R}^{d} \times \mathbb{R}$. Given $2n$ observations $\mathcal{D} = \{(Z_i, Y_i)\}_{i=1}^{2n}$ drawn i.i.d. from the above model, the goal is to estimate the following linear functional 
\begin{align}
\label{eq:theta-general}
    \theta_0 = \theta(\mu_0) \qquad \text{with} \qquad \theta(\mu) = \mathbb{E}_{}[\psi(\mu, Z)]
\end{align} for a known operator $\psi: \mathcal{H}_\mu \times \mathbb{R}^{d+1} \to \mathbb{R}$. We focus on the setting under which $\theta(\cdot)$ is linear and the map $\mu \mapsto \psi(\mu,\cdot)$ is Lipschitz from $L_2(\nu_z)$ to $L_2(\nu_z)$, as formalized below.
\begin{condition}[Linearity and continuity] \label{cond:linear-functional} There exists a constant $\const \ge 1$ such that the following condition holds for $\theta(\cdot)$ on $\mathcal{H}_\mu$: 
\begin{itemize}
\item[(a)] For any $\mu, \tilde{\mu} \in \mathcal{H}_\mu$ and $\alpha \in \mathbb{R}$, $\theta(\alpha \mu + \tilde{\mu}) = \alpha \cdot \theta(\mu) + \theta(\tilde{\mu})$;  
\item[(b)] For any $\mu, \tilde{\mu} \in \mathcal{H}_\mu$, $\mathbb{E}[|\psi(\mu, Z) - \psi(\tilde{\mu}, Z)|^2] \le \const \mathbb{E}[|\mu(Z) - \tilde{\mu}(Z)|^2]$.
\end{itemize}
\end{condition}
\cref{cond:linear-functional} implies that $\theta$ is a linear continuous functional. Indeed, $\theta(0) = 0$ by the linearity of $\theta(\cdot)$, and \cref{cond:linear-functional} (b) then implies that, for any $\mu\in \mathcal{H}_\mu$,
\begin{align*}
    |\theta(\mu)| = \left|\mathbb{E}[\psi(\mu, Z) - \psi(0, Z)]\right| \le \left\{\mathbb{E}[|\psi(\mu, Z) - \psi(0, Z)|^2]\right\}^{1/2} \le \const^{1/2} \|\mu\|_{L_2(\nu_z)}.
\end{align*}
The Riesz representer theorem therefore yields a unique Riesz representer $r_0 \in \mathcal{H}_\mu$ for the linear functional $\theta(\cdot)$ such that
\begin{align}
\label{eq:riesz-representer}
    \forall \mu \in \mathcal{H}_\mu, \qquad \theta(\mu) = \mathbb{E} \left[\psi(\mu, Z)\right] = \mathbb{E} \left[\mu(Z) \cdot r_0(Z)\right].
\end{align} The partial linear model in \cref{sec:estimator} is an instance of the above general framework with $\mathcal{H}_\mu = \{\mu(X, T) = \beta \cdot T + f(X): \beta \in \mathbb{R}, f \in L_2(\nu_x)\}$, $\psi(\mu, Z) = \beta$ and Riesz representer $r_0(X, T) = (T - \mathbb{E}[T|X]) / \mathbb{E}[(T - \mathbb{E}[T|X])^2]$.

\myparagraph{Method and main results.} In a similar spirit to Section 4 of \cite{gu2026optimally}, \cref{algo-tas-lf} extends the \texttt{O}-branch in \cref{algo-tas} by initializing the debiasing weights with an arbitrary estimate $\hat{r}$ of the Riesz representer and then fine-tuning these weights against the reference class $\mathcal{G}_0$. 

\begin{algorithm}[H]
\caption{\our~for linear functional $\theta(\cdot)$, $\bo$-branch}
\label{algo-tas-lf}
\begin{algorithmic}[1]
\State \textbf{Input:} $\mathcal{D}_1 = \{(Z_i, Y_i)\}_{i=1}^n$; $\hat{\mu}, \hat{r}$ independent of $\mathcal{D}_1$, debiasing reference class $\mathcal{G}_0$, hyper-parameter $\lambda$.
\State Initialize debiasing weights $a \in \mathbb{R}^n$ with $a_i \gets \tilde{a}_i$ where $\tilde{a}_i = \hat{r}(Z_i)$. 
\State Fine-tune weights $a$ by adversarial estimation on $(\mathcal{G}_0, \mathcal{D}_1)$ and get fine-tuned weights $\hat{a}$:
\begin{align*}
    \hat{a} \in \argmin_{a\in \mathbb{R}^n} \frac{\lambda}{n} \sum_{i=1}^n |a_i - \tilde{a}_i|^2+ \sup_{f \in \partial \mathcal{G}_0} \left|\frac{1}{n} \sum_{i=1}^{n} \{\psi(f, Z_i) - a_i \cdot f(Z_i)\} \right| - \frac{1}{n}\sum_{i=1}^n f^2(Z_i)
\end{align*} 
\State \textbf{Output:} The final estimate $\hat{\theta} \gets \frac{1}{n} \sum_{i=1}^n \{\psi(\hat{\mu}, Z_i) + (Y_i - \hat{\mu}(Z_i)) \cdot \hat{a}_i\}$. 
\end{algorithmic}
\end{algorithm}

\begin{remark}[Implementing the initial estimate] The algorithm can be applied to any initial black-box estimates $\hat{\mu}, \hat{r}$ of the outcome function $\mu_0$ and the Riesz representer $r_0$. For example, $\hat{\mu}$ may be obtained by fitting $Y$ on $Z$ via least squares over a hypothesis class $\mathcal{G}_\mu$ on $\mathcal{D}_2 = \{(Z_i, Y_i)\}_{i=n+1}^{2n}$. The Riesz representer $\hat{r}$ may be estimated either by plugging in the estimated nuisance functions into a known analytic expression or directly via adversarial Riesz estimation as in \cite{chernozhukov2026adversarial}. 
\end{remark}

The following conditions allow the \bo-branch oracle-type inequality, i.e., \eqref{eq:plm-branch-o} in \cref{thm:plm2}, to extend to the general linear functional $\theta$. 
\begin{condition} \label{cond:lf-all} There exists a constant $\const \ge 1$ such that:
\begin{itemize}
    \item[(a)] (Boundedness) $|\varepsilon| \le \const$, $\|r_0\|_\infty \lor \|\mu_0\|_\infty \le \const$, $\|\psi(\mu_0, \cdot)\|_\infty \le \const$. 
    \item[(b)] (Function class $\mathcal{G}_0$) The reference class $\mathcal{G} \subset \mathcal{H}_\mu$ satisfies $\mathcal{G}_0 \in \mathcal{HCS}(\delta_n^\stoc; n, \nu_z)$ and $\hat{\mu} \in \mathcal{G}_0$. Moreover, $\|\psi(f, \cdot)\|_\infty \le \const$ for any $f\in \partial \mathcal{G}_0$, and
    \begin{align*}
        \sup_{\delta \ge \delta_n^\stoc} \frac{\mathsf{R}_{n, \nu_z}(\delta; (\partial \mathcal{G}_0)^\psi)}{\const \cdot \delta} \le \delta_n^\stoc \qquad \text{for} \qquad (\partial \mathcal{G}_0)^\psi := \{z\mapsto \psi(g - \tilde{g}, z): g, \tilde{g} \in \mathcal{G}_0\}. 
    \end{align*} 
    \item[(c)] (Initial black-box estimate) $\hat{r}$, $\hat{\mu}$ are independent of $\mathcal{D}_1$ and uniformly bounded, i.e., $\|\hat{r}\|_\infty \lor \|\hat{\mu}\|_\infty \le \const$. 
\end{itemize}
\end{condition}
\begin{theorem} \label{thm:lf-o}
    Under \cref{cond:linear-functional} and \cref{cond:lf-all}, if $\|\hat{r} - r_0\|_{L_2(\nu_z)} \le \delta_r$ for $\delta_r \in \mathbb{R}^+$ and we choose $\lambda \asymp [\delta_n^\stoc / (\delta_r + \delta_n^\stoc)]^2$, then for any $t \ge 1$, the estimate returned by \cref{algo-tas-lf} satisfies, with probability at least $1-2e^{-t}-e^{-n(\delta_n^\stoc)^2}$,
    \begin{align}\label{eq:lf-o}
    \begin{split}
        \tilde{C}^{-1} |\hat{\theta} - \theta_0 - \hat{\varphi}_{\mathsf{LF}, \theta}| &\le \inf_{\bar{g} \in \mathcal{G}_0} \myblue{\left(\delta_r + \delta_n^\stoc \right) \cdot \|\bar{g} - \mu_0\|_2} + \myred{\|\bar{g} - \hat{\mu}\|_2^2 + (\delta_n^\stoc)^2 }\\
        &\qquad \qquad + \mygray{\left(\delta_r + \|\bar{g} - \mu_0\|_2 + \delta_n^\stoc + \sqrt\frac{t}{n}\right) \cdot \sqrt{\frac{t}{n}}}
    \end{split}
    \end{align} where $\tilde{C} = \poly(\const)$ and $\hat{\varphi}_{\mathsf{LF}, \theta} = \frac{1}{n} \sum_{i=1}^n \psi(\mu_0; Z_i) - \theta_0 + \varepsilon_i r_0(Z_i)$ is the empirical average of an influence function for $\theta_0$. Under the additional condition $\mathbb E[\varepsilon^2\mid Z]\equiv\sigma_\varepsilon^2$, this influence function is semi-parametrically efficient.
\end{theorem}

\begin{remark}[Two-learner specialization] \label{remark:lf-two-learner} 
Suppose $\hat{\mu}$ is an empirical least squares estimator over a hypothesis class $\mathcal{G}_\mu$ and we choose $\mathcal{G}_0 = \mathcal{G}_\mu$ in \cref{algo-tas-lf}. Define $\delta_\mu^\appr := \inf_{g\in \mathcal{G}_\mu} \|g - \mu_0\|_{L_2(\nu_z)}$, and $\delta_\mu := \delta_\mu^\appr + \delta_n^\stoc$. If both $\hat{\mu}$ and $\hat{r}$ can consistently estimate $\mu_0$ and $r_0$, respectively, then \eqref{eq:lf-o} yields
\begin{align}
\label{eq:error-lf-two-learner}
    |\hat{\theta} - \theta_0 - \hat{\varphi}_{\mathsf{LF}, \theta}| \lesssim \myblue{\delta_r \cdot \delta_\mu^\appr} + \myred{(\delta_\mu)^2} + \mygray{o(n^{-1/2})}
\end{align} Recall that DML has the nuisance remainder $\delta_\mu \cdot \delta_r + o(n^{-1/2})$ \citep{chernozhukov2018double, chernozhukov2026adversarial}. Thus, \our~is no worse up to constants when $\delta_\mu \le \delta_r$. Moreover, when $\delta_\mu \ll \delta_r$, i.e., the outcome function is easier to estimate, the use of \our~estimator gives strict improvement over the DML nuisance remainder when under-smoothing on $\mathcal{G}_\mu$ is adopted, i.e., $\delta_\mu^\appr \ll \delta_\mu \asymp \delta_n^\stoc$. This yields a strictly faster overall rate whenever $\delta_\mu \cdot \delta_r \gg n^{-1/2}$. 
\end{remark}

\begin{remark}[Asymptotic normality] \label{remark:lf-normality} Continue with the two-learner specialization in \cref{remark:lf-two-learner}, and suppose $\delta_r \cdot \delta_\mu^\appr + (\delta_\mu)^2 = o(n^{-1/2})$. Define $\sigma_0 := [\mathrm{Var}(\psi(\mu_0, Z_i) + \varepsilon \cdot r_0(Z))]^{1/2} > 0$, and the estimated influence values $\hat{\varphi}_i := \psi(\hat{\mu}, Z_i) - \hat{\theta} + (Y_i - \hat{\mu}(Z_i)) \cdot \hat{a}_i$. Then we have 
\begin{align*}
    \frac{\sqrt{n}(\hat{\theta} - \theta_0)}{\hat{\sigma}} \overset{d}{\to} \mathcal{N}(0, 1) \qquad \text{for} \qquad \hat{\sigma} = \sqrt{\frac{1}{n} \sum_{i=1}^n \hat{\varphi}_i^2} ~~\text{satisfying} ~~ |\hat{\sigma} - \sigma_0| = o_P(1).
\end{align*} 
\end{remark}

\subsection{The \texttt{R}-branch for ATE estimation}
\label{subsec:r-branch}

Let $\mathcal{D} = \{(X_i, T_i, Y_i)\}_{i=1}^{2n}$ be $2n$ i.i.d. observations generated from \begin{align}
\label{eq:model-ate}
\begin{split}
    T &= \pi_0(X) + u ~~~ \qquad \text{with} ~~\mathbb{E}[u|X] = 0, T\in\{0, 1\}, \\
    Y &= \mu_0(X, T) + \varepsilon \qquad \text{with} ~~\mathbb{E}[\varepsilon|X, T] = 0,
\end{split}
\end{align} where $X\in \mathcal{X} \subseteq \mathbb{R}^d$, and $\pi_0(x) = \mathbb{P}(T=1|X=x)$ is the propensity score. The target is the linear functional $\theta_0 = \mathbb{E} \left[ \mu_0(X, 1) - \mu_0(X, 0)\right]$. Under the usual consistency, unconfoundedness, and overlap assumptions, $\theta_0$ equals the average treatment effect in the potential outcome framework

\cref{algo-tas-ate-r} is motivated by and extends the central construction of \cite{wang2024debiasd}. Their debiased inverse propensity weighting (DIPW) estimator edits observation-specific outcome weights through cross-sample linear moment constraints under a sparse logistic propensity model. We replace these linear constraints with adversarial moment calibration over a generic reference class $\mathcal{G}_0$, allowing the same principle to be applied with generic black-box propensity learners. When $\hat{m}_i = \tilde{m}_i$, the estimator in Step 4 reduces algebraically to the usual AIPW estimator based on $(\hat\mu, \hat\pi)$. 

\begin{algorithm}
\caption{\our~for ATE, $\br$-branch}
\label{algo-tas-ate-r}
\begin{algorithmic}[1]
\State \textbf{Input:} Data split: $\mathcal{D}_1 = \{(X_i, T_i, Y_i)\}_{i=1}^n$, $\mathcal{D}_2 = \{(X_i, T_i, Y_i)\}_{i=n+1}^{2n}$; $\hat{\mu}, \hat{\pi}$ independent of $\mathcal{D}_1$, debiasing reference class $\mathcal{G}_0$, hyper-parameter $\lambda$.
\State Initialize debiasing weights $m\in \mathbb{R}^n$ with $m_i \gets \tilde{m}_i$ where $\tilde{m}_i = (1 - \hat{\pi}(X_i)) \hat{\mu}(X_i, 1) + \hat{\pi}(X_i) \hat{\mu}(X_i, 0)$. 
\State Fine-tune weights $m$ by adversarial estimation on $(\mathcal{G}_0, \mathcal{D}_1, \mathcal{D}_2)$ and get fine-tuned weights $\hat{m}$:
\begin{align*}
    \hat{m} \in &\argmin_{m\in \mathbb{R}^n} \frac{\lambda}{n} \sum_{i=1}^n |m_i - \tilde{m}_i|^2+ \sup_{g \in \mathcal{G}_0} \Phi_{\ate, \pi}(g, m) ~~\text{with} \\
    & \Phi_{\ate, \pi}(g, m) = \left|\frac{1}{n} \sum_{i=n+1}^{2n} \hat{Y}_i \cdot \Delta_g(X_i) - \frac{1}{n} \sum_{i=1}^n m_i \cdot \Delta_g(X_i) \right| - \frac{1}{n}\sum_{i=1}^n (g - \hat{\pi})^2(X_i),
\end{align*} where the pseudo-outcome $\hat{Y}_i$ and direction $\Delta_g$ are defined by $\hat{Y}_i = \frac{1 - \hat{\pi}(X_i)}{\hat{\pi}(X_i)} Y_i T_i + \frac{\hat{\pi}(X_i)}{1 - \hat{\pi}(X_i)} Y_i (1 - T_i)$ and $\Delta_g(X_i) = (g - \hat{\pi})(X_i) / [\hat{\pi}(X_i) (1 - \hat{\pi}(X_i))]$.
\State Assign $\hat{\theta} \gets \frac{1}{n} \sum_{i=1}^n (Y_i - \hat{m}_i) \left(\frac{T_i}{\hat{\pi}(X_i)} - \frac{1 - T_i}{1 - \hat{\pi}(X_i)}\right)$. 
\State \textbf{Output:} The final estimate $\hat{\theta}$. 
\end{algorithmic}
\end{algorithm}

The analysis of \cref{algo-tas-ate-r} uses two conditions. \cref{cond:ate} imposes boundedness and overlap on the data-generating process, while \cref{cond:ate-g} controls the reference class and the initial estimates.

\begin{condition}[Boundedness and overlap] \label{cond:ate} There exists a constant $\const \ge 2$ such that $\|\mu_0\|_\infty \lor |\varepsilon| \le \const$, and $\pi_0(x) \in [1/\const, 1-1/\const]$  for any $x\in \mathcal{X}$.
\end{condition}
\begin{condition} \label{cond:ate-g} Let $\nu_x$ be the distribution of $X$, the following holds for a constant $\const \ge 2$:
\begin{itemize}
    \item[(a)] The reference class satisfies $\mathcal{G}_0 \in \mathcal{HCS}(\delta_n^\stoc; n, \nu_x)$ and $\hat{\pi} \in \mathcal{G}_0$. Moreover, $g(x) \in [1/\const, 1-1/\const]$ for any $x\in \mathcal{X}$ and $g\in \mathcal{G}_0$. 
    \item[(b)] The initial estimates $(\hat{\mu}, \hat{\pi})$ are independent of $\mathcal{D}_1 = \{(X_i, T_i, Y_i)\}_{i=1}^n$, and $\|\hat{\mu}\|_\infty\le \const$. 
\end{itemize}
\end{condition}

\cref{thm:ate-r} gives the resulting oracle-type inequality for the estimator in \cref{algo-tas-ate-r}.

\begin{theorem} \label{thm:ate-r} Under \cref{cond:ate} and \cref{cond:ate-g}, if the sum of first-stage estimation errors $\|\hat{\pi} - \pi_0\|_{L_2(\nu_x)} + \sup_{w\in \{0, 1\}} \|\hat{\mu}(\cdot, w) - \mu_0(\cdot, w)\|_{L_2(\nu_x)} \le \delta_{all}$ for $\delta_{all} \in \mathbb{R}^+$ and we choose $\lambda \asymp [\delta_n^\stoc / (\delta_{all} + \delta_n^\stoc)]^2$, then for any $t\in [1, n]$, the estimate $\hat{\theta}$ returned by \cref{algo-tas-ate-r} satisfies, with probability at least $1-3e^{-t} - e^{-n(\delta_n^\stoc)^2}$, 
\begin{align}
\label{eq:ate-error}
\begin{split}
    \tilde{C}^{-1} |\hat{\theta} - \theta_0 - \hat{\varphi}_{\ate}| &\le \inf_{\bar{g} \in \mathcal{G}_0} \myblue{(\delta_{all} + \delta_n^\stoc) \cdot \|\bar{g} - \pi_0\|_2} + \myred{\|\bar{g} - \hat{\pi}\|_2^2 + \|\pi_0 - \hat{\pi}\|_2^2 + (\delta_n^\stoc)^2} \\
    &\qquad \qquad + \mygray{\sqrt{\frac{t}{n}} \left[\delta_{all} + \delta_n^\stoc + \|\bar{g} - \pi_0\|_2 + \|\hat{\pi} - \pi_0\|_2^{1/2} + \left(\frac{t}{n}\right)^{1/4}\right]}
\end{split}
\end{align} where $\tilde{C} = \poly(\const)$ and $\hat{\varphi}_{\ate} = \frac{1}{n} \sum_{i=1}^n \{\mu_0(X_i, 1) - \mu_0(X_i, 0)\} + \varepsilon_i \cdot u_i / \{\pi_0(X_i) \cdot (1 - \pi_0(X_i))\} - \theta_0$ is the empirical average of the semiparametric efficient influence function for the ATE.
\end{theorem}

\cref{thm:ate-r} also yields direct analogues of \cref{remark:lf-two-learner} and \cref{remark:lf-normality}. Suppose that $\hat\pi$ is an empirical least squares estimator over $\mathcal G_\pi=\mathcal G_0$, and define
$\delta_\pi^{\mathrm{appr}}:=\inf_{g\in\mathcal G_\pi}\|g-\pi_0\|_{L_2(\nu_x)}$ and $\delta_\pi:=\delta_\pi^{\mathrm{appr}}+\delta_n^{\mathrm{stoc}}$. 
If the first stage errors are consistent, then \eqref{eq:ate-error} yields
\begin{align*}
\left|\hat\theta-\theta_0-\hat\phi_{\mathrm{ATE}}\right|\lesssim\myblue{\delta_{\mathrm{all}} \cdot \delta_\pi^{\mathrm{appr}}}+\myred{\delta_\pi^2}+\mygray{o(n^{-1/2})}.
\end{align*}
Consequently, asymptotic normality follows when the displayed nuisance remainder is $o(n^{-1/2})$.

%% file: 3-proof.tex
\section{Proof of Theorem \ref{thm:main-plm}}
\label{sec:proof-main}

We first introduce notation used only in the proof of \cref{thm:main-plm}. We let $\|f\|_2 = \sqrt{\mathbb{E}[f^2(X)]}$ and $\|f\|_n = \sqrt{\frac{1}{n} \sum_{i=1}^n f^2(X_i)}$, and adopt $\lesssim$ to omit the $\poly(\const)$ factors. The following proposition characterizes the estimation error of the initial estimate $(\hat{\beta}, \hat{\mu}, \hat{\pi})$, which is a restatement of Proposition S2.3 in \cite{gu2026optimally}. The proof is standard; one can directly refer to the proof therein. 

\begin{proposition}
\label{prop:plm-fact}
    The following event occurs with probability at least $1-e^{-n(\delta_n^\stoc)^2}$ for some constant $\tilde{C}_0 = \poly(\const)$,
    \begin{align*}
        \mathcal{C}_0 := \left\{|\hat{\beta} - \beta_0| + \|\hat{\mu} - \mu_0\|_2  \le \tilde{C}_0 \left(\delta_\mu^\appr + \delta_n^\stoc\right),  \|\hat{\pi} - \pi_0\|_2  \le \tilde{C}_0 \left(\delta_\pi^\appr + \delta_n^\stoc\right) \right\}.
    \end{align*}
\end{proposition}

Without loss of generality, we consider the case where $\delta_\mu^\appr \ge n^{-100}$ and $\delta_\pi^\appr \ge n^{-100}$. Otherwise, define $\bar{\delta}_\mu^\appr \gets \delta_\mu^\appr + n^{-100}$ and $\bar{\delta}_\pi^\appr \gets \delta_\pi^\appr + n^{-100}$, we can prove the corresponding branch-wise bound with \(\bar\delta_h^{\mathrm{appr}}\) in place of \(\delta_h^{\mathrm{appr}}\), using the facts that (1) $\bar\delta_h^{\mathrm{appr}}$ remains an upper bound on the corresponding approximation error; and (2) $\lambda \ge (\delta_n^\stoc)^2 \ge \log(n)/n$. 

We will prove in the two subsequent subsections
\begin{align*}
    |\hat{\theta}_\mu - \theta_0 - \hat{\varphi}_{\mathsf{PLM}}| &\lesssim \delta^\star + \lambda (\delta_\pi^\appr)^2 + \frac{\delta_n^\stoc \cdot \delta_\mu^\appr}{\sqrt{\lambda}} ~~~~\text{if}~~~~ \delta_\mu^\appr \le \delta_\pi^\appr \\
    |\hat{\theta}_\pi - \theta_0 - \hat{\varphi}_{\mathsf{PLM}}| &\lesssim \delta^\star + \lambda (\delta_\mu^\appr)^2 + \frac{\delta_n^\stoc \cdot \delta_\pi^\appr}{\sqrt{\lambda}} ~~~~\text{if}~~~~ \delta_\mu^\appr > \delta_\pi^\appr
\end{align*} where $\delta^\star = \delta_\pi^\appr \cdot \delta_\mu^\appr + (\delta_n^\stoc)^2 + \sqrt{\frac{t}{n}} \cdot \left[\delta_\mu^\appr + \delta_\pi^\appr + \frac{\delta_n^\stoc}{\sqrt{\lambda}} + \sqrt{\frac{t}{n}}\right]$. Combining the separate error bounds with the estimator choice criterion \eqref{eq:est1-est} concludes the proof. 

\subsection{Proof of {\sc Case 1}}

Denote $\mathsf{H}_\mu(a) = \sup_{\beta, f\in \partial \mathcal{G}} \Phi_\mu(f, \beta, a)$ for any $a\in \mathbb{R}^n$ and $\bar{D} = \frac{1}{n} \sum_{l=1}^n T_l (T_l - \pi_0(X_l))$. Define $\bar{a} \in \mathbb{R}^n$ be such that $\bar{a}_i = (T_i - \pi_0(X_i)) / \bar{D}$ if $\bar{D} \neq 0$ and $\bar{a}_i \equiv 0$ otherwise. With a slight abuse of notation, we will also use $\|a - a'\|_n = \{\frac{1}{n} \sum_{i=1} (a_i - a_i')^2\}^{1/2}$ for $a, a' \in \{\bar{a}, \tilde{a}, w, \hat{a}\} \subseteq \mathbb{R}^n$. We first pick $\bar{g} \in \mathcal{G}$ such that $\|\bar{g} - \mu_0\|_2 \le 2\delta_\mu^\appr$ by the definition of the approximation error in \eqref{eq:error-appr-g}. Here, we first define the high-probability event used to derive our error rate, which follows directly from the standard empirical process arguments and concentration inequalities. 

\begin{proposition}
\label{prop:plm-fact-mu}
    Under the event $\mathcal{C}_0$ in \cref{prop:plm-fact}, assume $\delta_n^\stoc \tilde{C}_0 \le 1$ for some $\tilde{C}_0 = \poly(\const)$. For any $t\in [1, n]$, the following event occurs with probability at least $1-e^{-n(\delta_n^\stoc)^2} - e^{-t}$ for some $\tilde{C}_1 = \poly(\const)$:
    \begin{align}
        \mathsf{H}_\mu(\bar{a}) &\le \tilde{C}_1 [\delta_n^\stoc]^2, \label{eq:fact-1-h}\\
        \|\bar{g} - \hat{\mu}\|_n^2 & \le \tilde{C}_1 \left[(\delta_\mu^\appr + \delta_n^\stoc)^2 + (t/n)\right], \label{eq:fact-1-g}\\
        \sup_{i\in [n]} |\bar{a}_i - w_i| &\le \tilde{C}_1 \cdot (t/n)^{1/2}, \label{eq:fact-1-w}\\
        \|\bar{a} - \tilde{a}\|_n &\le \tilde{C}_1 \left[(t/n)^{1/2} + \delta_\pi^\appr + \delta_n^\stoc \right], \label{eq:fact-1-a}\\
        \max\left\{\sqrt{\frac{n}{t}}\cdot \left|\frac{1}{n} \sum_{i=1}^n (\mu_0 - \bar{g}) (X_i)\cdot w_i \right|, \|\bar{g} - \mu_0\|_n\right\} &\le \tilde{C}_1 \left[\delta_\mu^\appr + \sqrt{\frac{t}{n}}\right]. \label{eq:fact-1-mu}
    \end{align}
\end{proposition}

Let $\mathcal{C}_1$ be the event in \cref{prop:plm-fact-mu}. The rest of the proof will proceed conditioned on the event $\mathcal{C}_0\cap \mathcal{C}_1$. On \(\mathcal C_1\), \eqref{eq:fact-1-h} implies \(\bar D\neq0\), since \(\mathsf{H}_\mu(0)=\infty\). Given that $\tilde{a}$ is uniformly bounded by its construction and \cref{cond:reg-plm}, $\lambda \|\bar{a} - \tilde{a}\|_n^2 + \mathsf{H}_\mu(\bar{a}) < \infty$, thus $\bar{a}$ has finite objective value. Since $\mathsf{H}_\mu(a)$ is lower semi-continuous and $\lambda \|a - \tilde{a}\|_n^2$ is coercive, the objective in \eqref{eq:est1-case1-step2} has a finite minimum. Because $\hat{a}$ in \eqref{eq:est1-case1-step2} minimizes the objective therein, this means
\begin{align}
    \lambda \|\tilde{a} - \hat{a}\|_n^2 + \mathsf{H}_\mu(\hat{a}) &\le \lambda \|\tilde{a} - \bar{a}\|_n^2 + \mathsf{H}_\mu(\bar{a}) \nonumber \\
    &\overset{(a)}{\lesssim} \lambda \left[(\delta_\pi^\appr)^2 + \frac{t}{n}\right] + (\delta_n^\stoc)^2  \label{eq:plm-mu-minimizer}.
\end{align} Here $(a)$ follows from the error bound of $\mathsf{H}_\mu(\bar{a})$ \eqref{eq:fact-1-h}, $\|\tilde{a} - \bar{a}\|_n$ \eqref{eq:fact-1-a} in \cref{prop:plm-fact-mu}, and the condition $\lambda \le 1$ in the statement. Since $\mathsf{H}_\mu(\hat{a}) \ge \Phi_\mu(0, 0, \hat{a}) = 0$, $\|\hat{a} - \tilde{a}\|_n^2 \lesssim (\delta_\pi^\appr)^2 + \frac{t}{n} + \frac{(\delta_n^\stoc)^2}{\lambda}$. Recall the definition of $w_i$ in \eqref{eq:psi-plm}, with a slight abuse of notation $\|\hat{a} - w\|_n = \sqrt{\frac{1}{n} \sum_{i=1}^n (\hat{a}_i - w_i)^2}$, combining the triangle inequality with the error bound of $\|\tilde{a} - \bar{a}\|_n$ \eqref{eq:fact-1-a} and $\|\bar{a} - w\|_\infty$ \eqref{eq:fact-1-w}, we find
\begin{align} \label{eq:plm-mu-hata-error}
    \|\hat{a} - w\|_n &\le \|\hat{a} - \tilde{a}\|_n + \|\tilde{a} - \bar{a}\|_n + \|\bar{a} - w\|_n  \lesssim \delta_\pi^\appr + \sqrt{\frac{t}{n}} + \frac{\delta_n^\stoc}{\sqrt{\lambda}}.
\end{align}

On the other hand, combining the non-negativity of $\lambda \cdot \|\tilde{a} - \hat{a}\|_n^2$ and \eqref{eq:plm-mu-minimizer} also gives
\begin{align}
\label{eq:plm-mu-h}
    \mathsf{H}_\mu(\hat{a}) \lesssim \lambda \left[(\delta_\pi^\appr)^2 + \frac{t}{n}\right] + (\delta_n^\stoc)^2  .
\end{align} Recall the definition of $\mathsf{H}_\mu(\hat{a})$ and $\Phi_\mu(f, \beta, a)$, this implies the following instance-dependent error bound by letting $\beta=0$ in the supremum:
\begin{align}
\label{eq:plm-ins}
    \forall g, \tilde{g} \in \mathcal{G}, \qquad \left|\frac{1}{n} \sum_{i=1}^n (g - \tilde{g})(X_i) \cdot \hat{a}_i \right| \le \|g - \tilde{g}\|_n^2 + \mathsf{H}_\mu(\hat{a}).
\end{align} On the other hand, let $f=0$ in the $\sup$ operator, we obtain $|\frac{1}{n} \sum_{i=1}   T_i \hat{a}_i \cdot \beta - \beta| \le \mathsf{H}_\mu(\hat{a})$ for any $\beta \in \mathbb{R}$, the finiteness of $\mathsf{H}_\mu(\hat{a})$ implies the normalization constraint $\frac{1}{n} \sum_{i=1} \hat{a}_i \cdot T_i = 1$. Recall our choice of fixed $\bar{g} \in \mathcal{G}$ satisfying $\|\bar{g} - \mu_0\|_2 \le 2\delta_\mu^\appr$, it follows from the definition of our estimator in \eqref{eq:est1-case1-step3} the model \eqref{eq:intro:model}, and the normalization constraint $\frac{1}{n} \sum_{i=1} \hat{a}_i \cdot T_i = 1$ that
\begin{align*}
    \hat{\theta}_\mu - \theta_0 
    &\overset{(a)}{=} \underbrace{\frac{1}{n} \sum_{i=1}^n \left(\varepsilon_i + \mu_0(X_i) - \bar{g}(X_i) \right) \cdot \hat{a}_i}_{\mathsf{T}_1} + \underbrace{\frac{1}{n} \sum_{i=1}^n (\bar{g}(X_i) - \hat{\mu}(X_i)) \cdot \hat{a}_i}_{\mathsf{T}_2}.
\end{align*} 

We will handle $\mathsf{T}_1$ and $\mathsf{T}_2$ separately. For $\mathsf{T}_2$, we plug in the instance-dependent error bound \eqref{eq:plm-ins} with $g \gets \bar{g}$ and $\tilde{g} \gets \hat{\mu}$ and obtain
\begin{align*}
    |\mathsf{T}_2| = \left|\frac{1}{n} \sum_{i=1}^n (\bar{g}(X_i) - \hat{\mu}(X_i)) \cdot \hat{a}_i\right| &\lesssim \|\bar{g} - \hat{\mu}\|_n^2 + \left[(\delta_n^\stoc)^2 + \lambda \cdot \left((\delta_\pi^\appr)^2 + \frac{t}{n} \right)\right] \\
    &\lesssim (\delta_n^\stoc)^2 + (\delta_\mu^\appr)^2 + \frac{t}{n} + \lambda \cdot (\delta_\pi^\appr)^2
\end{align*} by the error bound of $\|\hat{\mu} - \bar{g}\|_n$ in \eqref{eq:fact-1-g}, the error bound of $\mathsf{H}_\mu(\hat{a})$ in \eqref{eq:plm-mu-h}, and the fact that $\lambda \le 1$. 

On the other hand, we will further decompose $\mathsf{T}_1$ into
\begin{align*}
    \mathsf{T}_1 &= \frac{1}{n} \sum_{i=1}^n \varepsilon_i w_i + \frac{1}{n} \sum_{i=1}^n \varepsilon_i (\hat{a}_i - w_i) + \frac{1}{n} \sum_{i=1}^n (\mu_0(X_i) - \bar{g}(X_i)) \cdot (\hat{a}_i - w_i) \\
    &\qquad \qquad + \frac{1}{n} \sum_{i=1}^n (\mu_0(X_i) - \bar{g}(X_i)) \cdot w_i.
\end{align*} 

We will use the following concentration result: it follows from the fact that $\{\varepsilon_i\}_{i=1}^n$ are zero-mean random variables conditioned on fixed $\mathcal{D}_2$ and $\{(X_i, T_i)\}_{i=1}^n$, under which $\{\hat{a}_i - w_i\}_{i=1}^n$ are fixed because our construction of the weights only depends on $\mathcal{D}_2$ and $\{(X_i, T_i)\}_{i=1}^n$ in the first split $\mathcal{D}_1$.
\begin{lemma}
\label{lemma:concentration1}
    Under the setting of \cref{thm:main-plm}, assume further that $\|\hat{a} - w\|_n \le \delta$. There exists some constant $\tilde{C}_1 = \poly(\const)$ such that the following event occurs with probability at least $1-e^{-t}$ for any $t\ge 1$:
    \begin{align*}
        \left|\frac{1}{n} \sum_{i=1}^n \varepsilon_i \left(\hat{a}_i - w_i\right) \right| &\le \tilde{C}_1  \delta \cdot \sqrt{\frac{t}{n}}
    \end{align*}
\end{lemma} 

Let $\mathcal{C}_2$ be the event in \cref{lemma:concentration1}. We proceed with the rest of the proof conditioned on the event $\mathcal{C}_0 \cap \mathcal{C}_1 \cap \mathcal{C}_2$. It then follows from the triangle inequality, Cauchy-Schwarz inequality, the error bound of $\|\hat{a} - w\|_{n}$ in \eqref{eq:plm-mu-hata-error} and the error bounds in \eqref{eq:fact-1-mu} and \cref{lemma:concentration1} that
\begin{align*}
    \left|\mathsf{T}_1 - \frac{1}{n} \sum_{i=1}^n \varepsilon_i w_i \right| &\lesssim  \|\bar{g} - \mu_0\|_n \cdot \|\hat{a} - w\|_n +  \sqrt{\frac{t}{n}} \cdot \left[\delta_\pi^\appr + \frac{\delta_n^\stoc}{\sqrt{\lambda}} + \delta_\mu^\appr + \sqrt{\frac{t}{n}} \right] \\
    &\lesssim \left[\sqrt{\frac{t}{n}} + \delta_\mu^\appr \right] \cdot \left[\delta_\pi^\appr + \frac{\delta_n^\stoc}{\sqrt{\lambda}} + \sqrt{\frac{t}{n}} \right] 
\end{align*} 

Putting all the pieces together, under the intersection of events in \cref{prop:plm-fact-mu}, \cref{prop:plm-fact}, and \cref{lemma:concentration1} that occurs with probability at least $1-2(e^{-n(\delta_n^\stoc)^2} + e^{-t})$ by the union bound, we have
\begin{align*}
    \left|\hat{\theta}_\mu - \theta_0 - \hat{\varphi}_{\mathsf{PLM}}\right| &\le |\mathsf{T}_2| + \left|\mathsf{T}_1 - \frac{1}{n} \sum_{i=1}^n \varepsilon_i w_i \right| \\
    %&\lesssim (\delta_n^\stoc)^2 + (\delta_\mu^\appr)^2 + \lambda (\delta_\pi^\appr)^2 + \left(\sqrt{\frac{t}{n}} + \delta_\mu^\appr \right) \cdot \left(\delta_\pi^\appr + \frac{\delta_n^\stoc}{\sqrt{\lambda}} + \sqrt{\frac{t}{n}} \right) \\
    &\lesssim \delta_\mu^\appr \cdot \delta_\pi^\appr + (\delta_n^\stoc)^2 + \lambda (\delta_\pi^\appr)^2 + \frac{\delta_\mu^\appr \cdot \delta_n^\stoc}{\sqrt{\lambda}} + \sqrt{\frac{t}{n}}\left(\delta_\pi^\appr + \frac{\delta_n^\stoc}{\sqrt{\lambda}} + \sqrt{\frac{t}{n}}\right)
\end{align*} by noting $\delta_\mu^\appr \le \delta_\pi^\appr$ and $\lambda \le 1$. This concludes the proof for {\sc Case 1}. \qed

\subsection{Proof of {\sc Case 2}}

Denote $m_0(x) = \mathbb{E}[Y|X=x] = \beta_0 \cdot \pi_0(x) + \mu_0(x)$ and $\xi := Y - m_0(X) = \beta_0 u + \varepsilon$. Let $\tilde{m}(x) = \hat{\beta} \cdot \hat{\pi}(x) + \hat{\mu}(x)$.  With a slight abuse of notation, we will also denote $\|m - f\|_n = \{\frac{1}{n} \sum_{i=1}^n |m_i - f(X_i)|^2\}^{1/2}$ for $m\in \{\bar{m}, \hat{m}\}$ and $f\in L_2(\nu_x)$ and use $\tilde{m} \in \mathbb{R}^n$ or $\tilde{m} \in L_2(\nu_x)$ exchangeably based on the context. 

We also pick $\bar{g} \in \mathcal{G}$ such that $\|\bar{g} - \pi_0\|_2 \le 2 \delta_\pi^\appr$. The proof is based on the following instance-dependent error bounds \eqref{eq:fact-2-gxi} -- \eqref{eq:fact-2-g2} and concentration results \eqref{eq:fact-2-m} -- \eqref{eq:fact-2-u2}. The concentration results are all standard and follow directly from the Bernstein-type inequality. 

\begin{proposition} \label{prop:plm-fact-pi}
    Denote $\mathbb{P}_{n, l}[f]= \frac{1}{n} \sum_{i=1+(l-1)\cdot n}^{l\cdot n} f(X_i, T_i, Y_i)$ be the empirical average operator on the split $\mathcal{D}_l$ and $\mathbb{P}[f] = \mathbb{E}[f(X, T, Y)]$, and define $\delta_\pi := \delta_\pi^\appr + \delta_n^\stoc + (t/n)^{1/2}$ only for the proof of \cref{thm:main-plm}. 
    The following event occurs with probability at least $1-e^{-n(\delta_n^\stoc)^2} - e^{-t}$ for some $\tilde{C}_1 = \poly(\const)$:
    \begin{align}
        \forall f \in \partial \mathcal{G}, l\in\{1, 2\} \qquad \qquad \left|\mathbb{P}_{n, l} [f(X) \cdot \xi] \right| &\le \tilde{C}_1 \left[\|f\|_2 \cdot \delta_n^\stoc + (\delta_n^\stoc)^2 \right]\label{eq:fact-2-gxi}\\
         \left| \left(\mathbb{P}_{n, l} - \mathbb{P} \right)[m_0(X) f(X)] \right| &\le \tilde{C}_1 \left[\|f\|_2 \cdot \delta_n^\stoc + (\delta_n^\stoc)^2 \right]\label{eq:fact-2-gm0}\\
        \left|\|f\|_2 - \|f\|_{n}\right| &\le \tilde{C}_1 \delta_n^\stoc \label{eq:fact-2-g2} \\
        \|\tilde{m} - m_0\|_n &\le \tilde{C}_1 \left[\delta_\pi + \delta_\mu^\appr \right] \label{eq:fact-2-m} \\
        \|\bar{g} - \pi_0\|_n &\le \tilde{C}_1 \left[ \delta_\pi^\appr + (t/n)^{1/2} \right] \label{eq:fact-2-gbar} \\
        \left|\mathbb{P}_{n,1} \left[(\hat{\pi} - \pi_0)(X) \cdot \varepsilon\right] \right| + \left|\mathbb{P}_{n,1} \left[(\hat{\pi} - \pi_0)(X) \cdot u\right] \right| &\le \tilde{C}_1 \delta_\pi \cdot (t/n)^{1/2} \label{eq:fact-2-pi} \\
        \|\pi_0 - \hat{\pi}\|_n^2 &\le \tilde{C}_1 \delta_\pi^2 \label{eq:fact-2-pi2} \\
        \left|\mathbb{P}_{n, 1} \left[(T - \hat{\pi}(X))^2\right] - \mathbb{E}[u^2] \right| &\le \tilde{C}_1 \left[\delta_\pi^2 + (t/n)^{1/2} \right]\label{eq:fact-2-u2}
    \end{align}
\end{proposition}

Let $C_1=\poly(\const)$ be the constant $\tilde{C}_1$ in \cref{prop:plm-fact-pi}, and the event in \cref{prop:plm-fact-pi} be $\mathcal{C}_1$. The rest of the proof proceeds conditioned on the event $\mathcal{C}_0 \cap \mathcal{C}_1$. 
Denote $\mathsf{H}_\pi(m) = \sup_{f \in \partial \mathcal{G}} \Phi_\pi(f, m)$ and $\bar{m} = (m_0(X_1), \ldots, m_0(X_n)) \in \mathbb{R}^n$. We first claim that under the inequalities \eqref{eq:fact-2-gxi} -- \eqref{eq:fact-2-g2}, 
\begin{align}
    \label{eq:plm-h-pi}
    \mathsf{H}_\pi(\bar{m}) \lesssim (\delta_n^\stoc)^2,
\end{align} whose proof can be found at the end of this section. 
Given $\hat{m}$ minimizes the objective in \eqref{eq:est1-case2-step2} and $\mathsf{H}_\pi(\hat{m}) \ge \Phi_\pi(0, \hat{m}) = 0$, we obtain $\lambda \cdot \|\hat{m} - \tilde{m}\|_n^2 \le \mathsf{H}_\pi(\bar{m}) + \lambda \cdot \|\bar{m} - \tilde{m}\|_n^2$. Combining it with the triangle inequality further yields
\begin{align}
    \|\hat{m} - m_0\|_n &\le \|\hat{m} - \tilde{m}\|_n + \|\tilde{m} - m_0\|_n \nonumber \\
    &\le 2\|\tilde{m} - m_0\|_n + \sqrt{\frac{\mathsf{H}_\pi(\bar{m})}{\lambda}} \lesssim  \delta_\mu^\appr + \frac{\delta_n^\stoc}{\sqrt{\lambda}} + \sqrt{\frac{t}{n}}, \label{eq:plm-m-error}
\end{align} where the last inequality follows from the error bound \eqref{eq:fact-2-m} of $\|m_0 - \tilde{m}\|_n$ and the facts that $\lambda \le 1$ and $\delta_\pi^\appr \le \delta_\mu^\appr$. We further argue the following instance-dependent error bound by the inequalities \eqref{eq:fact-2-gxi} -- \eqref{eq:fact-2-g2}, the proof is also deferred to the end of the section. 
\begin{align}
\label{eq:plm-pi-ins}
    \forall f\in \partial \mathcal{G}, \qquad \left|\frac{1}{n} \sum_{i=1}^n (m_0(X_i) - \hat{m}_i)\cdot f(X_i) \right| \lesssim \|f\|_2^2 + (\delta_n^\stoc)^2 + \lambda (\delta_\mu^\appr)^2 + \frac{t}{n}
\end{align}

We assume without loss of generality that $2\const\cdot C_1 \left\{(\delta_\pi^\appr + \delta_n^\stoc)^2 + \sqrt{\frac{t}{n}}\right\} \le 1$, otherwise the error rate is trivial because $\hat{\theta}_\pi$ is bounded. This is because the denominator in $\hat{\theta}_\pi$ is bounded from below. At the same time, as $\|\hat{m}\|_n \le \|\hat{m} - m_0\|_n + \|m_0\|_n \lesssim 1$ provided $\lambda \ge (\delta_n^\stoc)^2$, the numerator is also bounded by applying the triangle inequality and Cauchy-Schwarz inequality. 

Denote $\hat{D} = \frac{1}{n} \sum_{i=1}^n \left[T_i - \hat{\pi}(X_i) \right]^2$, note $|\hat{D} - \mathbb{E}[u^2]| \le C_1 \left\{(\delta_\pi^\appr + \delta_n^\stoc)^2 + (t/n)^{1/2}\right\}$ by \eqref{eq:fact-2-u2}. As long as $2\const\cdot C_1 \left\{(\delta_\pi^\appr + \delta_n^\stoc)^2 + (t/n)^{1/2}\right\} \le 1$, $\hat{D} \ge 1/(2\const)$ hence the truncation in \eqref{eq:est1-case2-step3} is inactive and its denominator equals $\hat{D}$. By our constructed estimator \eqref{eq:est1-case2-step3} and the model \eqref{eq:intro:model}, we decompose
\begin{align*}
    &\frac{1}{n} \sum_{i=1}^n [T_i - \hat{\pi}(X_i)]^2\left[\hat{\theta}_\pi - \theta_0 \right]
    %&~~= \frac{1}{n} \sum_{i=1}^n \left\{Y_i - (T_i - \hat{\pi}(X_i)) \beta_0 - \hat{m}_i\right\} (T_i - \hat{\pi}(X_i)) \\
    %&~~= \frac{1}{n} \sum_{i=1}^n \left\{\varepsilon_i + m_0(X_i) + \beta_0 u_i - \beta_0 u_i - (\pi_0(X_i) - \hat{\pi}(X_i)) \beta_0 - \hat{m}_i\right\} (T_i - \hat{\pi}(X_i)) \\
    = \frac{1}{n} \sum_{i=1}^n \left\{\varepsilon_i + m_0(X_i) - \hat{m}_i - (\pi_0(X_i) - \hat{\pi}(X_i)) \beta_0 \right\} \left\{u_i + \pi_0(X_i) - \hat{\pi}(X_i)\right\}
\end{align*}
Recall our choice of fixed $\bar{g} \in \mathcal{G}$ such that $\|\bar{g} - \pi_0\|_2 \le 2\delta_\pi^\appr$, it further follows from the triangle inequality that
\begin{align*}
    \left|\hat{\theta}_\pi - \theta_0 - \hat{\varphi}_{\mathsf{PLM}} \right| &\le \left|\frac{1}{\mathbb{E}[u^2]} - \frac{1}{\hat{D}} \right| \cdot \left|\frac{1}{n} \sum_{i=1}^n \varepsilon_i u_i \right| + \frac{1}{|\hat{D}|} \cdot \left|\frac{1}{n} \sum_{i=1}^n (m_0(X_i) - \hat{m}_i) \cdot u_i \right| \\
    &\qquad \qquad + \frac{|\beta_0|}{|\hat{D}|} \frac{1}{n} \sum_{i=1}^n |\hat{\pi}(X_i) - \pi_0(X_i)|^2 + \frac{|\beta_0|}{|\hat{D}|} \cdot \left|\frac{1}{n} \sum_{i=1}^n (\pi_0(X_i) - \hat{\pi}(X_i)) \cdot u_i\right| \\
    &\qquad \qquad + \frac{1}{|\hat{D}|} \cdot \left| \frac{1}{n} \sum_{i=1}^n (\pi_0 - \hat{\pi})(X_i) \varepsilon_i\right| + \frac{1}{|\hat{D}|} \left|\frac{1}{n}\sum_{i=1}^n (m_0(X_i) - \hat{m}_i) (\pi_0 - \hat{\pi})(X_i) \right|
\end{align*} It then follows from the error bounds \eqref{eq:fact-2-pi}, \eqref{eq:fact-2-pi2}, \eqref{eq:fact-2-u2} and the fact that $|\hat{D}|^{-1} + |\beta_0| \lesssim 1$ by the above discussion and \cref{cond:reg-plm} that
\begin{align*}
    \left|\hat{\theta}_\pi - \theta_0 - \hat{\varphi}_{\mathsf{PLM}} \right| &{\lesssim} \left[(\delta_\pi^\appr + \delta_n^\stoc)^2 + \sqrt\frac{t}{n} \right] \cdot \left|\frac{1}{n} \sum_{i=1}^n \varepsilon_i u_i\right| + \left|\frac{1}{n} \sum_{i=1}^n (m_0(X_i) - \hat{m}_i) \cdot u_i \right| \\
    &\qquad \qquad + (\delta_\pi^\appr + \delta_n^\stoc)^2 + \frac{t}{n} + \left|\frac{1}{n} \sum_{i=1}^n \left(m_0(X_i) - \hat{m}_i \right)(\pi_0 - \bar{g})(X_i) \right| \\
    &\qquad \qquad + \left|\frac{1}{n} \sum_{i=1}^n (m_0(X_i) - \hat{m}_i) (\bar{g}(X_i) - \hat{\pi}(X_i)) \right|.
\end{align*}
We will further use the following concentration result, which is due to our construction of the cross-sample debiasing weights $\{\hat{m}_i\}_{i=1}^n$ that depend on $\mathcal{D}_2$ and only $\{X_i\}_{i=1}^n$ in the first sample $\mathcal{D}_1$. This will make $\{u_i \cdot (m_0(X_i) - \hat{m}_i) \}_{i=1}^n$ be independent zero-mean random variables given fixed $\mathcal{D}_2, \{X_i\}_{i=1}^n$. 
\begin{lemma}
\label{lemma:concentration2}
Under the setting of \cref{thm:main-plm}, assume further that $\|\hat{m} - m_0\|_n \le \delta$, then with probability at least $1-e^{-t}$,
\begin{align*}
    \left|\frac{1}{n} \sum_{i=1}^n \varepsilon_i \cdot u_i \right| \le \tilde{C}_1 \sqrt{\frac{t}{n}} \qquad \text{and} \qquad \left|\frac{1}{n} \sum_{i=1}^n (m_0(X_i) - \hat{m}_i) \cdot u_i\right| \le \tilde{C}_1 \cdot \delta  \sqrt{\frac{t}{n}}.
\end{align*}
\end{lemma}

Let $\mathcal{C}_2$ be the event in \cref{lemma:concentration2}. We proceed with the rest of the proof conditioned on the event $\mathcal{C}_0 \cap \mathcal{C}_1 \cap \mathcal{C}_2$, which occurs with probability at least $1-2(e^{-n(\delta_n^\stoc)^2} + e^{-t})$ by the union bound. Substituting the error bound in \cref{lemma:concentration2} with $\delta \gets \delta_\mu^\appr + \frac{\delta_n^\stoc}{\sqrt{\lambda}} + \sqrt{\frac{t}{n}}$ by \eqref{eq:plm-m-error}, we arrive at
\begin{align}
\label{eq:plm-theta-derror}
\begin{split}
    \left|\hat{\theta}_\pi - \theta_0 - \hat{\varphi}_{\mathsf{PLM}} \right| &\lesssim (\delta_\pi^\appr + \delta_n^\stoc)^2 + \frac{t}{n} + \sqrt{\frac{t}{n}} \cdot \left(\delta_\mu^\appr + \frac{\delta_n^\stoc}{\sqrt{\lambda}} +\delta_n^\stoc\right) + \mathsf{T}_{1} + \mathsf{T}_2.
\end{split} 
\end{align} where $\mathsf{T}_1 = \left|\frac{1}{n} \sum_{i=1}^n \left(m_0(X_i) - \hat{m}_i \right)(\pi_0 - \bar{g})(X_i) \right|$ and $\mathsf{T}_2 = \left|\frac{1}{n} \sum_{i=1}^n (m_0(X_i) - \hat{m}_i) (\bar{g}(X_i) - \hat{\pi}(X_i)) \right|$.
On the one hand, it follows from the Cauchy-Schwarz inequality and the error bound for $\|\hat{m} - m_0\|_n$ in \eqref{eq:plm-m-error} and the error bound for $\|\bar{g} - \pi_0\|_n$ in \eqref{eq:fact-2-gbar} that
\begin{align*}
    \mathsf{T}_1 \le \|\hat{m} - m_0\|_n \cdot \|\bar{g} - \pi_0\|_n \lesssim \left[\delta_\pi^\appr + \sqrt{\frac{t}{n}} \right] \cdot \left[\delta_\mu^\appr + \frac{\delta_n^\stoc}{\sqrt{\lambda}} + \sqrt{\frac{t}{n}}\right].
\end{align*} On the other hand, applying the instance-dependent error bound \eqref{eq:plm-pi-ins} with $f  = \bar{g} - \hat{\pi} \in \partial \mathcal{G}$ and the fact that $\|\hat{\pi} - \bar{g}\|_2 \lesssim \|\hat{\pi} - \pi_0\|_2 + \|\pi_0 - \bar{g}\|_2 \lesssim \delta^\stoc_n + \delta_\pi^\appr$ under the event in \cref{prop:plm-fact} yield
\begin{align*}
    \mathsf{T}_2 \lesssim \|\bar{g} - \hat{\pi}\|_2^2 + (\delta_n^\stoc)^2 + \lambda (\delta_\mu^\appr)^2 + t/n \lesssim (\delta^\appr_\pi)^2 + (\delta_n^\stoc)^2 + \lambda (\delta_\mu^\appr)^2 + t/n.
\end{align*} 

Now putting the two inequalities back into \eqref{eq:plm-theta-derror} and keeping in mind that we are considering the case where $\delta_\pi^\appr \le \delta_\mu^\appr$ and $\lambda \le 1$, we can conclude that, with the high-order residual term $\mathsf{Res} = (t/n)^{1/2} \cdot \{\delta_\mu^\appr + \delta_n^\stoc/\sqrt{\lambda} + (t/n)^{1/2}\}$, 
\begin{align*}
    \left|\hat{\theta}_\pi - \theta_0 - \hat{\varphi}_{\mathsf{PLM}} \right| %&\lesssim (\delta_\pi^\appr + \delta_n^\stoc)^2 + \frac{t}{n} + \sqrt{\frac{t}{n}} \cdot \left(\delta_\mu^\appr + \frac{\delta_n^\stoc}{\sqrt{\lambda}} +\delta_n^\stoc\right) \\
    %&\qquad \qquad + \left[\delta_\pi^\appr + \sqrt{\frac{t}{n}} \right] \cdot \left[\delta_\mu^\appr + \frac{\delta_n^\stoc}{\sqrt{\lambda}} + \sqrt{\frac{t}{n}}\right] + \lambda (\delta_\mu^\appr)^2 \\
    &\lesssim \delta_\mu^\appr \cdot \delta_\pi^\appr + (\delta_n^\stoc)^2 + \lambda (\delta_\mu^\appr)^2 + \frac{\delta_\pi^\appr \cdot \delta_n^\stoc}{\sqrt{\lambda}} + \mathsf{Res}.
\end{align*} 
This concludes the proof for {\sc Case 2}. \qed

\begin{proof}[Proof of \eqref{eq:plm-h-pi}]
To this end, it follows from the triangle inequality that, for any $f = g - \tilde{g} \in \partial \mathcal{G}$, 
\begin{align*}
    \Phi_\pi(f, \bar{m}) &= \left|\frac{1}{n} \sum_{i=n+1}^{2n} Y_i f(X_i) - \frac{1}{n} \sum_{i=1}^n m_0(X_i) f(X_i)\right| - \|f\|_n^2 \\
    &\le \left|\frac{1}{n} \sum_{i=n+1}^{2n} \xi_i f(X_i)\right| + \left|\frac{1}{n} \sum_{i=n+1}^{2n} m_0(X_i) f(X_i) - \mathbb{E}[m_0(X) f(X)]\right| \\
    &\qquad \qquad + \left|\mathbb{E}[m_0(X) f(X)] - \frac{1}{n} \sum_{i=1}^n m_0(X_i) f(X_i)\right| - \|f\|_n^2\\
    &\overset{(a)}{\le} 3 C_1 \left\{\|f\|_2 \cdot \delta_n^\stoc + (\delta_n^\stoc)^2 \right\} - \|f\|_n^2 \\
    &\overset{(b)}{\le} 3 C_1 \left\{\left(\|f\|_n + C_1 \delta_n^\stoc \right) \cdot \delta_n^\stoc + (\delta_n^\stoc)^2 \right\} - \|f\|_n^2 \\
    &\overset{(c)}{\le} \|f\|_n^2 + \frac{(3C_1)^2}{4} (\delta_n^\stoc)^2 + 3C_1(C_1 + 1) \cdot (\delta_n^\stoc)^2 - \|f\|_n^2 \lesssim (\delta_n^\stoc)^2
\end{align*} where we apply \eqref{eq:fact-2-gxi} and \eqref{eq:fact-2-gm0} in (a), \eqref{eq:fact-2-g2} in (b), and use a variant of Young's inequality $xy \le x^2 + \frac{y^2}{4}$ with $x= \|f\|_n$ and $y = (3C_1) \delta_n^\stoc$ in (c). Taking the supremum on both sides concludes the proof.
\end{proof}
\begin{proof}[Proof of the inequality \eqref{eq:plm-pi-ins}] By the optimality of $\hat{m}$, and after dropping the non-negative term $\lambda \|\hat{m} - \tilde{m}\|_n^2$, we find
\begin{align} \label{eq:plm-pi-h2}
    \mathsf{H}_\pi(\hat{m}) \le \mathsf{H}_\pi(\bar{m}) + \lambda \cdot \|m_0 - \tilde{m}\|_n^2 \lesssim (\delta_n^\stoc)^2 + \lambda \left[(\delta_\mu^\appr)^2 + \frac{t}{n} \right]
\end{align} where again we plug in the upper bound \eqref{eq:fact-2-m} for $\|\tilde{m} - m_0\|_n$ and use the facts $\delta_\pi^\appr \le \delta_\mu^\appr$, $\lambda \le 1$ to simplify the presentation. It follows from the triangle inequality that, for any $f = g - \tilde{g} \in \partial \mathcal{G}$, 
\begin{align*}
    \left|\frac{1}{n} \sum_{i=1}^n (m_0(X_i) - \hat{m}_i) f(X_i)\right| &\le \left|\frac{1}{n} \sum_{i=n+1}^{2n} Y_i f(X_i) - \frac{1}{n} \sum_{i=1}^n \hat{m}_i f(X_i)\right| \\
    &\qquad \qquad + \left|\frac{1}{n} \sum_{i=n+1}^{2n} Y_i f(X_i) - \frac{1}{n} \sum_{i=1}^n m_0(X_i) f(X_i)\right| \\
    &\overset{(a)}{=} \Phi_\pi(f, \hat{m}) + \|f\|_n^2 \\
    &\qquad \qquad + \left|\frac{1}{n} \sum_{i=n+1}^{2n} (m_0(X_i) + \xi_i) \cdot f(X_i) - \frac{1}{n} \sum_{i=1}^n m_0(X_i) f(X_i)\right| \\
    &\overset{(b)}{\le} \mathsf{H}_\pi(\hat{m}) + \|f\|_n^2 + \left|\frac{1}{n} \sum_{i=n+1}^{2n} \xi_i f(X_i) \right| \\
    &\qquad \qquad + \sum_{l\in \{1, 2\}} \left|\frac{1}{n} \sum_{i=1+(l-1)n}^{ln} m_0(X_i) \cdot f(X_i) - \mathbb{E}[m_0(X) f(X)] \right| 
\end{align*} Here (a) follows from the definition of $\Phi_\pi$, (b) follows from the definition $\mathsf{H}_\pi(m) = \sup_{f \in \partial \mathcal{G}} \Phi_\pi(f, m)$ and the triangle inequality. Now we plug in the error bound \eqref{eq:plm-pi-h2} for $\mathsf{H}_\pi(\hat{m})$ as well as the instance-dependent error bounds \eqref{eq:fact-2-gxi} and \eqref{eq:fact-2-gm0}, 
\begin{align*}
    \left|\frac{1}{n} \sum_{i=1}^n (m_0(X_i) - \hat{m}_i) f(X_i)\right| &~ \lesssim \|f\|_n^2 + (\delta_n^\stoc)^2 + \lambda \left[(\delta_\mu^\appr)^2 + \frac{t}{n} \right] + \|f\|_2 \cdot \delta_n^\stoc + (\delta_n^\stoc)^2 \\
    &\overset{(a)}{\lesssim} \|f\|_2^2 + (\delta_n^\stoc)^2 + \lambda \left[(\delta_\mu^\appr)^2 + \frac{t}{n} \right],
\end{align*} where (a) follows from the equivalence between $L_2$ norm and its empirical counterparts in \eqref{eq:fact-2-g2} and Young’s inequality. This concludes the proof for \eqref{eq:plm-pi-ins}.
\end{proof}